\documentclass{article}
\usepackage[utf8]{inputenc}
  
\usepackage{fullpage}
  \usepackage{amsfonts,amsmath,amssymb,mathrsfs,amsthm, comment
}

\usepackage{authblk} 
\usepackage{setspace} 
\usepackage{geometry} 
 \usepackage[normalem]{ulem}
\usepackage{microtype}
\usepackage{enumitem}
\usepackage{color}
\usepackage{tikz}
\usepackage{graphicx}
\usepackage{xcolor}
\usepackage{empheq}
\usepackage{textgreek}
\usepackage{url}

\definecolor{green1}{rgb}{0.0, 0.5, 0.0}

\graphicspath{{./figure/}}

\usepackage{hyperref}

\newtheorem{definition}{Definition}[section]

\newtheorem{thm}[definition]{Theorem}

\newtheorem{lemma}[definition]{Lemma}
\newtheorem{prop}[definition]{Proposition}

 \newtheorem{assumption}{Assumption}

\theoremstyle{remark}
\newtheorem{example}[definition]{Example}
 \newtheorem{remark}[definition]{Remark}

\newcommand{\R}{\mathbb{R}}
\newcommand{\N}{\mathbb{N}}

\newcommand{\G}{{\Gamma}}

\newcommand{\cX}{\mathcal{X}}
\newcommand{\cB}{\mathcal{B}}
\newcommand{\cC}{\mathcal{C}}
\newcommand{\cJ}{\mathcal{J}}
\newcommand{\cD}{\mathcal{D}}

\newcommand{\cF}{\mathcal{F}}

\newcommand{\cH}{\mathcal{H}}
\newcommand{\cI}{\mathcal{I}}

\newcommand{\cL}{\mathcal{L}}

\newcommand{\cU}{\mathcal{U}}

\newcommand{\wt}{\widetilde}

\newcommand{\koo}{{k\to\infty}}

\newcommand{\wto}{\rightharpoonup}
\newcommand{\weaks}{\stackrel{\star}{\wto}}

\newcommand{\cA}{\mathcal{A}}

\newcommand{\ubar}{\overline{u}}
\newcommand{\loc}{{\rm loc}}

\newcommand{\dis}{\displaystyle}
\newcommand{\Wl}{W_{\rm loc}}

\newcommand{\LQR}[2]{\mathrm{n\text{-}QR}_{#1}\!\left(#2\right)}
\newcommand{\seq}{{\rm seq}}

\DeclarePairedDelimiterX{\inp}[2]{\langle}{\rangle}{#1, #2}

\newcommand{\AI}{{\cA}}

\newcounter{altassumption}[assumption]
\renewcommand{\thealtassumption}{\theassumption w}

\allowdisplaybreaks

\title{Pattern preservation in finite to infinite-horizon optimal control problems for dissipative systems}

\author[1]{Matteo Della Rossa\thanks{Corresponding author. Email: \texttt{matteo.dellarossa@polito.it}}}
\author[2]{Lorenzo Freddi\thanks{Email: \texttt{lorenzo.freddi@uniud.it}}}
\author[3]{Thiago Alves Lima\thanks{Email: \texttt{ thiago.lima@gp.ita.br}}}
\affil[1]{Department of Electronics and Telecommunications, Politecnico di Torino, Corso Duca degli Abruzzi 24, 10129 Torino, Italy}
\affil[2]{Dipartimento di Scienze Matematiche, Informatiche e Fisiche, Università di Udine, via delle Scienze 206, 33100 Udine, Italy}
\affil[3]{Systems and Energy Engineering Division, Aeronautics Institute of Technology (ITA), Fortaleza, 60415-513, CE, Brazil}

\date{} 
\begin{document}
\maketitle
 \begin{abstract}
  This paper focuses on infinite-horizon optimal control problems for dissipative systems and the relations to their finite-horizon formulations. We show that, for a large class of problems, dissipativity of the state equation, when a coercive storage function exists,  implies that  infinite-horizon optimal controls can be obtained as limits of the corresponding finite-horizon ones. This property is referred to as \emph{pattern preservation}, or \emph{pattern-preserving property.}
 Our analysis  establishes a formal link between dissipativity theory and the variational convergence framework in optimal control, thus providing a concrete and numerically tractable condition for verifying  pattern preservation. Numerical examples illustrate the effectiveness and limitations of the proposed sufficient conditions.
\end{abstract}

 \section{Introduction}
 
In optimal control theory, a fundamental dichotomy arises between finite-horizon and infinite-horizon control problems.

Informally, in the finite-horizon case, given a state dynamics (typically, a controlled ODE) and a cost functional (typically, an integral cost), one seeks an optimal control input over a compact interval of time $[0,T]$. In contrast, the infinite-horizon problem requires minimizing (or maximizing) the cost over the unbounded interval $[0,\infty)$.
Although this distinction may appear minor or merely notational at first glance, it often leads to substantially different theoretical and computational challenges. For instance, Pontryagin's Maximum Principle, possibly combined with second-order conditions, frequently provides the structure of the optimal control in finite-horizon problems. In the infinite-horizon setting, however, the terminal boundary condition for the co-state is no longer available and can only be recovered in a weaker form. A seminal reference on necessary optimality conditions for infinite-horizon control problems is~\cite{Halkin74}. More recent developments can be found in~\cite{BasCassFra18,CANNFra18,Faulwasser21}.  Conversely, Hamilton–Jacobi–Bellman (HJB) equations tend to be simpler to analyze in the infinite-horizon case. For autonomous problems, where both the dynamics and the cost are time-invariant, the value function is time-independent, and the HJB equation becomes stationary. For finite-horizon problems, instead, the HJB equations are coupled with terminal boundary conditions, and the value function generally depends explicitly on time. For a  formal introduction to optimal control theory and the mentioned results, we refer to the monographs~\cite{Vinter2010Book,Clarke13}.

It is therefore natural to investigate  the connections between these two frameworks, and specifically to understand if and when the infinite-horizon problem  can be seen as the limit of the corresponding finite-horizon problems, in a sense to be specified. A classical and technically significant example in which this idea is fruitful is the  linear-quadratic regulator (LQR) problem (see~\cite[Chapter 6]{Liberzon12}), in which a linear state equation is coupled with a quadratic positive cost. In this setting, it can be shown that not only does the infinite-horizon optimal cost coincide with the limit of finite-horizon costs as $T\to \infty$, but also that the  infinite-horizon optimal control can be obtained as the limit of finite-horizon optimal controls. This is a consequence of the convergence of (solutions to)  finite-horizon Riccati equations to the (solution to) algebraic Riccati equation in the limit. We refer to~\cite[Section 6.2]{Liberzon12} for a didactical and at the same time self-contained presentation of this argument.

Establishing, in a general framework, the convergence of finite-horizon optimal controls to the infinite-horizon solution is the focus of the recent work by~\cite{DellaFreddi25}. In such reference and also in this manuscript, this property is formalized and referred to as the
\emph{pattern-preserving property}.
The formal definition of this property is deferred to Definition~\ref{defn:PatternPreserving}. For the purposes of this introduction, it suffices to give an informal description. We say that the pattern-preserving property holds if the structural features of optimal controls observed on every finite horizon persist in the infinite-horizon limit. For instance, if for each final time $T>0$ the corresponding finite-horizon problem admits an optimal control $u_T$ with a prescribed piecewise structure (for instance, of bang–bang type), then the associated infinite-horizon problem also admits an optimal control exhibiting the same piecewise structure.

In~\cite{DellaFreddi25}, sufficient conditions implying such property are provided, relying essentially on  the equi-coercivity of the finite-horizon problems and their $\Gamma$-convergence toward the infinite-horizon problem as $T\to \infty$. For a formal introduction to $\Gamma$-convergence and its application to optimal control we refer to~\cite{DalMaso93,Braides02,ButtDalMas82, BelButFre93}. In~\cite{DellaFreddi25}, \emph{equi-coercivity} (roughly speaking, the property that sequences of control-state couples with bounded cost admit convergent subsequences) was postulated a priori. Only a few illustrative cases satisfying this property were presented, and, in general, equi-coercivity represented a condition that is difficult to verify both theoretically and numerically.

A parallel line of research has explored the relationship between \emph{dissipativity} in dynamical systems and optimal control. Dissipativity, formally introduced by Jan C. Willems in the early 1970s~\cite{Willems73,Willems1972b} and reviewed more recently in~\cite{BorMash06}, offers a unifying formulation for classical control concepts such as stability, passivity, input-to-state-stability, etc. At its core, a dissipative system is characterized by a storage function and a supply rate, providing a physically interpretable measure of how ``energy'' (in a broad sense) is stored or dissipated by a controlled ODE. This framework historically bridged frequency-domain techniques and state-space Lyapunov methods and it has become a classical topic in undergraduate (nonlinear) control programs.
In the context of optimal control, dissipativity has recently gained renewed relevance, particularly for infinite-horizon problems. Indeed, under some technical assumptions, it implies the \emph{turnpike phenomenon}, whereby optimal trajectories of long-horizon problems remain close to a steady-state optimum for the majority of the time. These behaviors can be rigorously captured using suitable inequalities on the storage function, with direct implications for model-predictive control (MPC) and computational strategies for large- or infinite-horizon problems. This relatively recent line of research has been developed in~\cite{Faulwasser21,FaulKorda17,Gruene2022,Hara2023,Trelat23, Angeli2024} and references therein.


In this paper, we establish a formal link between dissipativity and the \emph{equi-coercivity} (in the functional sense) property in optimal control problems. Specifically, we show that dissipativity together with a coercive storage function (i.e., with bounded sublevel sets), implies \emph{equi-coercivity} for a broad class of sequences of optimal control problems. This therefore provides two complementary contributions. First, a well-established and numerically friendly criterion for verifying equi-coercivity, addressing a gap that was present in \cite{DellaFreddi25}, as already pointed out. Second, when combined with a proof of $\Gamma$-convergence, this result shows that optimal control problems with dissipative state dynamics and a coercive storage function automatically satisfy the \emph{pattern-preserving property}.
This is consistent with the results presented in~\cite{ Faulwasser21,Gruene2022,Hara2023,Trelat23, Angeli2024,FaulGrune22}, where it is proved (in various context and under certain conditions) that (strict) dissipativity implies the turnpike property, i.e., that (finite- and infinite-horizon) optimal control-states couples remains essentially close to a steady-state of the dynamics. Thus, finite- and infinite-horizon controls substantially share the same structure, consistently with our notion of \emph{pattern-preservation}.
After the formal derivation of our results, we illustrate, with simple numerical examples, the strength and the limits of the proposed sufficient conditions.

The structure of this manuscript is as follows: in Section~\ref{sec:Preliminaries} we formally introduce the notion of pattern-preserving property, in a general variational formulation for optimal control problems. We then introduce the classes of problems studied in this manuscript, together with some preliminary  results. In Section~\ref{sec:ParticularCases} we present our main result, showing that dissipativity of the state-equation provides a crucial sufficient condition to establish the pattern-preserving property. In Section~\ref{Sec:Examples} we illustrate our findings with the help of elementary but insightful numerical examples. The consequences of our results in the context of non-linear and state-constrained \emph{quadratic regulator} problems are also presented.  Final comments and future perspective are sketched in the conclusive Section~\ref{sec:Conclusions}.\\
\textbf{Notation:} The symbol $\overline \R:=(-\infty,\infty]$ denotes the set of upper-extended real numbers.
By $|\cdot|$ we denote any norm in $\R^k$, $k\in\N$; when applied to matrices, it denotes the associated \emph{operator norm}. The usual norm in the Lebesgue space $L^p$ will be denoted by $\|\cdot\|_p$. Given $p\in[1,\infty]$, as usual, $p'$ denotes the conjugate exponent of $p$, defined by the equality $\frac{1}{p'}+\frac{1}{p}=1$ and the convention $1/\infty=0$.   The notation used for domain and codomain of spaces of functions is standard. Namely, $L^p(D,C)$ denotes the space of $L^p$ functions with domain $D$ and co-domain $C$; when  $C=\R$ we simply write $L^p(D)$.   Accordingly, 
by $\cD(\Omega,\R^k)$ we denote the space of vector valued test functions on the open set $\Omega$ and the dual  $\cD'(\Omega,\R^k)$ is the space of distributions.   Given a set $A$,  we define 
\[
\chi_A(x):=\begin{cases}
0&\text{if }x\in A\\
+\infty&\text{if }x\notin A
\end{cases}\ \mbox{ and }\ \mathbf{1}_A(x):=\begin{cases}
1&\text{if }x\in A\\
0&\text{if }x\notin A
\end{cases}
\]
the \emph{indicator} and \emph{characteristic}  functions of $A$, respectively.
We will also make use of the following standard definition (see, e.g., \cite[Definition 2.1.1]{Buttazzo89}):
a function $\ell:(0,\infty)\times \R^n\times \R^m\to (-\infty,+\infty]$  is said to be 
\begin{enumerate}[leftmargin=*]
\item an \emph{integrand} if it is measurable with respect to the Lebesgue-Borel measure $\cL\otimes \cB_n\otimes \cB_m$ on $(0,\infty)\times \R^n\times \R^m$;
\item a \emph{normal integrand}  if it is an integrand and $\ell(t,\cdot,\cdot)$ is lower semicontinuous for almost all $t\ge0$;
\item a \emph{normal convex integrand} if it is a normal integrand and, for almost all $t\ge0$, the map  $u\mapsto \ell(t,x,u)$ is convex for all $x\in \R^n$.
\end{enumerate}

\section{Preliminaries: pattern-preserving optimal control problems}\label{sec:Preliminaries}
In this section, we introduce the main definitions and some necessary preliminaries, in a general functional framework for optimal control.

Let $\cU$ and $\cX$ be two topological spaces, referred to as the spaces of \emph{controls} and \emph{states}, respectively. 
Consider a \emph{cost functional} $\cJ:\cU\times \cX\to \overline \R$ together with a set of \emph{admissible control--state pairs} $\cA\subset \cU\times \cX$. 
The corresponding \emph{optimal control problem} is formulated as
\begin{equation}\label{eq:AbstractOptContrProb0}
\begin{array}{c}
\displaystyle\inf_{\cU\times \cX}\cJ(u,x)\\[-1ex]
\ \\[-1ex]
\text{subject to}\\[-1ex]
\ \\[-1ex]
(u,x)\in \cA.
\end{array}
\end{equation}
Introducing the indicator function $\chi_{\cA}$ of the admissible set, problem~\eqref{eq:AbstractOptContrProb0} can be equivalently expressed as 
\begin{equation}\label{eq:AbstractOptContrProb}
\inf_{\cU\times \cX} (\cJ+\chi_{\cA}),
\end{equation}
so that the optimal control problem is reformulated as the minimization of the \emph{joint functional} $\cF:=\cJ+\chi_\cA$.  
With a slight abuse of terminology, we shall refer to this as \emph{the optimal control problem associated to $\cF$}.  
Moreover, $(u,x)\in \cU\times \cX$ is called a \emph{feasible couple} for $\cF$ if $\cF(u,x)<\infty$.
Similarly, a control $u\in \cU$ is called an \emph{optimal control} for $\cF$ if there exists $x\in \cX$ such that $\cF(u,x)=\inf_{\cU\times\cX}\cF$.
Next, we introduce a one-parameter \emph{family of minimization problems} corresponding to functionals $\cF_T:=\cJ_T+\chi_{\cA_T}$, where the parameter $T\ge0$ will play the role of (and can be interpreted as) the final time.  
Our goal is to analyze the asymptotic behavior of this family as $T\to \infty$. 

To this aim, we first state a crucial definition.

\begin{definition}\label{defn:EquiCoercivityGeneral}
A sequence of functionals $\cF_k:\cU\times \cX\to \overline \R$ is said to be
{\em sequentially equi-coercive w.r.t.\ $x$} if, for every $C>0$ and  every sequence $(u_k,x_k)$ such that $\cF_k(u_k,x_k)\le C$  and $u_k\to u$ in $\cU$, there exists a subsequence of $(x_k)$ converging to some  $x$ in $\cX$.
\end{definition}

Moreover, we recall the notion of minimizing sequence.
\begin{definition}\label{def_ms}
    A sequence $(u_k,x_k)$ 
in  $\cU\times \cX$ is said to be {\em minimizing} for the sequence $\cF_k:\cU\times \cX\to \overline \R$ if 
\begin{equation}\label{lili}
\liminf_{k\to\infty}\cF_{k}(u_k,x_k)=\liminf_{k\to\infty}\inf_{\cU\times \cX}\cF_{k}.
\end{equation} 
\end{definition}

\noindent It can be seen that minimizing sequences always exist, and moreover, any sequence of optimal pairs (if the optima exist) is trivially minimizing.

To move on in our investigation, we now make a specific choice of the control space $\cU$ (as topological space). Indeed, from now on we consider
\[
\cU=L^p((0,\infty),\R^m)
\]
with $p\in(1,\infty]$.
  From the topological point of view, we equip it with its weak$^\star$, or $\sigma (L^{p},L^{p'})$ topology coming from the duality $L^p=(L^{p'})^*$. We note that it is appropriate (also if unusual) to consider the weak$^\star$ topology also for $p\in (1,\infty)$, since, in such reflexive cases, it coincides with the {\em weak} one. For simplicity, we use this convention along the whole paper, and we simply use the symbol ``$\to$'' to denote such convergence.

We now provide the formal definition of the property under investigation, as introduced in~\cite[Definition  2.6]{DellaFreddi25}.

\begin{definition}[Pattern-preserving family]\label{defn:PatternPreserving}
 Let 
$\cF_T:\cU\times \cX\to \overline \R$, $T\in(0,\infty]$, be a parametrized family of optimal control problems and $0\leq T_k\to\infty$. The family $(\cF_T)$   
is said to be  $(T_k)$-{\em pattern preserving} if the following property is satisfied: 
 for any minimizing sequence  $(u_k,x_k)\in \cU\times \cX$ 
of $(\cF_{T_k})$ with 
\begin{enumerate}
\item     ${u_k}_{\vert_{[0,T_k]}}$ represented   in the form
\begin{equation}\label{eq:OptimalControlsPiecewise}
{u_k}_{\vert_{[0,T_k]}}=\sum_{j=1}^N u_j^k\mathbf{1}_{[\tau^k_{j-1},\tau^k_{j})}
\end{equation}
for suitable  $N\in \N\setminus\{0\}$ (independent of $k$)  , $u^k_1,\dots, u^k_N\in \cU$ and a partition $0=\tau^k_0\leq \tau_1^k\leq \dots\leq \tau_{N-1}^k\leq \tau^k_N=T_k$, 
\item and such that 
\begin{equation}\label{tjkujk}
\begin{array}{l}
\tau^{k}_j\to \tau^{\infty}_j\mbox{ in }\overline{\R},\\[1ex]
u_j^k\to u_j^\infty\mbox{ in }\cU,
\end{array}
\mbox{ for any }
j\in \{1,\dots, N\},
\end{equation} 
\end{enumerate}
it turns out that
\[
u_\infty:=\sum_{j=1}^{N} u_j^\infty\mathbf{1}_{[\tau^\infty_{j-1},\tau^\infty_{j})}\quad 
\text{ (with } [+\infty,+\infty):=\varnothing \text{)}
\]
is an optimal control for $\cF_\infty$.
\end{definition}

Although the definition might initially seem rather technical, it embodies an intuitive concept. When the parameter $T\in (0,\infty]$ represents the final time, in essence, it asserts that if all finite-horizon problems $\cF_{T_k}$ admit optimal controls $u_k$ 
exhibiting a piecewise structure similar to that in~\eqref{eq:OptimalControlsPiecewise}, 
then the limiting control problem $\cF_\infty$ 
also possesses an optimal control of the same type. 
The coefficients and partition points of this limiting control correspond to the limits of those associated with the sequence $(u_k)$. 
Hence, the definition provides a rigorous framework for ``transferring'' information about optimal controls from finite-horizon setting to the infinite-horizon case.

\begin{remark}[Equivalent formulation]\label{rem:equivalentFormulation}
It is worth emphasizing that the definition given above is, in principle, independent of the topology chosen on $\cU$, which may therefore be regarded as arbitrary. Throughout the remainder of this manuscript, however, we shall always endow $\cU=L^p((0,\infty),\R^m)$, with $p \in (0,\infty]$, with its weak$^\star$ topology, as already mentioned. In this setting, a more transparent  formulation of Definition~\ref{defn:PatternPreserving} can be stated for sequences of functionals that are \emph{future independent}.

We say that a functional $\cF_T:\cU\times \cX\to \overline \R$ is \emph{$T$-future independent}
if $\cF_T(u_1,x)=\cF_T(u_2,x)$ for all $(u_1,u_2,x)\in \cU^2\times \cX$ such that ${u_1}_{\vert_{(0,T)}}={u_2}_{\vert_{(0,T)}}$. In other words, $\cF_T$ is $T$-future independent whenever its value depends only on the restriction of the input variable $u$ to the interval $(0,T)$ and is therefore unaffected by the values of $u$ on $(T,\infty)$. This property is naturally satisfied by the class of functionals considered in this paper, since the parameter $T\in (0,\infty]$ represents the final time of the interval over which the minimization is performed. Under the above choice of topology on $\cU$, and assuming that each $\cF_T$ is $T$-future independent, Definition~\ref{defn:EquiCoercivityGeneral} is equivalent to the following more concise formulation:
\\
\emph{The family $(\cF_T)$   
is said to be  $(T_k)$-{\em pattern preserving} if the following property is satisfied: 
 for any minimizing sequence  $(u_k,x_k)\in \cU\times \cX$ 
of $(\cF_{T_k})$  such that  
\[
u_k\to u_\infty\text{ in } \cU, 
\]
it turns out that $u_\infty$ is an optimal control for $\cF_\infty$.}\\
The fact that Definition~\ref{defn:PatternPreserving} implies the above more concise formulation is immediate and does not need any further assumption, since the latter is simply the special case of the former corresponding to $N=1$.
The opposite implication, and thus the equivalence between these two formulations in the mentioned framework, follows  from the proof of Theorem~2.7 in~\cite{DellaFreddi25}.  This alternative characterization  provides yet another interpretation of Definition~\ref{defn:PatternPreserving}:  when the parameter $T$ is interpreted as the final time, the pattern-preserving property asserts that whenever the finite-horizon problems $\cF_{T_k}$ admit optimal controls $u_k$ and these controls converge in $\cU$ (with respect to the weak$^\star$ topology), the limiting control is necessarily optimal for the corresponding infinite-horizon problem.

By contrast, if a different topology is considered on $\cU$ and/or the family $(\cF_T)$ does not satisfy the future-independence property, the piecewise condition appearing in Definition~\ref{defn:PatternPreserving} becomes genuinely more general, even though its formulation is somewhat less immediate.
\end{remark}

Before presenting our main sufficient condition to ensure the pattern-preserving property for abstract control problems, we present a simple example for which such property is \emph{not} satisfied, somehow justifying the subsequent analysis.
\begin{example}[Non patter-preserving sequence of~OCPs]\label{ex:evans}
This (counter)-example is taken from~\cite[Example 1, pag. 5]{Evans24}.
Given a final time $T>0$, let us consider the problem of minimizing the functional
\[
J_T(u,x)=\int_0^T(u(t)-1)x(t)+\chi_{[0,1]}\big(u(t)\big)\,dt
\]
over the space $\cU=L^\infty(0,\infty)$ and under the scalar state-equation
\[
\begin{cases}
x'(t)=u(t)x(t),\\
x(0)=1.
\end{cases}
\]
As proved in~\cite[Section 4.4.2]{Evans24} (by using the Pontryagin principle), it can be seen that, for any $T>0$ an optimal control $u_T:(0,\infty)\to [0,1]$ for such problem is of the following bang-bang form:
\begin{equation*}\label{eq:uTevans}
u_T(t)=\begin{cases}
1\;\;\;\text{if } t<T-1,\\
0\;\;\;\text{if } t>T-1.
\end{cases}
\end{equation*}
Let us denote by $x_T$ the corresponding state-solution.
For any sequence $0\leq T_k\to \infty$, $(u_{T_k},x_{T_k})$ is in particular a minimizing sequence, since it is a sequence of optima. We further have that $u_{T_k}\to \bar u= \mathbf{1}_{(0,\infty)}$ (the constant function with value $1$ on $(0,\infty)$) weakly$^*$ in $\cU=L^\infty(0,\infty)$ and also uniformly on compact subsets of $(0,\infty)$.
On the other hand, such limiting control, for the infinite control problem, induces a zero cost: \[
J_\infty(\overline u,x)=\int_0^\infty 0\;dt=0,
\]
while any $1$-$0$ bang-bang control of the form $u_\tau\equiv \mathbf{1}_{(0,\tau)}$ for some $\tau\geq 0$ induces an unbounded negative cost (in particular this is the case for the constant control $u_0\equiv 0$). This implies that $\overline u$ is not an optimal control for the infinite-horizon problem, and thus the considered sequence of optimal control problems is \emph{not} pattern-preserving. 
\end{example}

We now present a statement providing  sufficient conditions for pattern-preserving property, which will be used in our following developments.
Indeed, one can prove that \emph{equi-coercivity w.r.t. $x$}  together with \emph{(sequential) $\Gamma$-convergence} constitute sufficient conditions to guarantee such a preservation of structure, as stated in what follows.

\begin{thm}\label{thm:PatternPreserving} Let $p\in (1,\infty]$ and  $\cU=L^p((0,\infty),\R^m)$ endowed with its weak$^\star$ topology. 
Consider a parametrized family of optimal control problems $\cF_T:\cU\times \cX\to \overline \R$, $T\in (0,\infty]$, and suppose that there exists a sequence $0\leq T_k\to \infty$ such that
\begin{enumerate}
    \item $\cF_{T_k}$ is sequentially equi-coercive w.r.t.\ $x$;
    \item $\displaystyle \G^-_{\rm seq}(\cU\times\cX)\lim_{k\to\infty}\cF_{T_k}=\cF_\infty$.
\end{enumerate} 
 Then $\cF_T$ is $(T_k)$-pattern preserving.
\end{thm}

\begin{proof} The proof proceeds along the same lines of \cite[Theorem 2.7]{DellaFreddi25} once having observed that for every minimizing sequence as in Definition \ref{defn:PatternPreserving} there exists also a bounded one with the same properties. Indeed, the sequence ${u_k}_{\vert_{[0,T_k]}}$ is bounded in $\cU$, because 
for every $j\in\{1,\dots,N\}$ the sequence  $k\mapsto u_j^k$ is bounded in~$\cU$ (being converging). Thus,  also the sequences  $k\mapsto v_j^k:=u_j^k\mathbf{1}_{[\tau^k_{j-1},\tau^k_{j})}$ are bounded, so implying the boundedness of the finite sum    ${u_k}_{\vert_{[0,T_k]}}=\sum_{j=1}^Nv_{j\vert_{[0,T_k]}}^k$. 
Nevertheless,  we can always modify such sequence (for instance,  by taking it equal to $0$ in $(T_k,\infty)$) in order to obtain a bounded minimizing sequence which satisfies again properties \eqref{tjkujk}. This means that, to guarantee pattern preservation, we can weaken the coercivity requirement with respect to $u$ made in \cite{DellaFreddi25} and ask only for coercivity w.r.t.\ $x$ as formulated in the statement.
\end{proof}

The formal definition of $\displaystyle \G^-_{\rm seq}$-convergence is recalled in~Appendix~\ref{sec:Appendix}.
We note that, for equi-coercivity and 
$\Gamma$-convergence to be meaningful and constructive, a precise definition of the state space $\cX$ and its topology is required. This will be provided in the following subsections.

\begin{remark}
For the family of control problems considered in Example \ref{ex:evans}, the equi-coercivity assumption could be satisfied 
by taking a weak enough topology on the state space $\cX$. Nevertheless, with respect to any reasonable choice of this kind, the $\G$-convergence assumption {\it 2.\ } cannot be satisfied.  Indeed, given $u_T\to\ubar$ as in the example, 
we correspondingly have
$$
x_T(t)=e^t 1_{[0,T-1]}(t)+e^{T-1} 1_{[T-1,+\infty)}(t)\to x^{\ubar}=e^t,
$$
and 
$$
\cF_T(u_T,x_T)=-e^{T-1}\to-\infty\quad\mbox{ and }\quad \cF_\infty(\ubar,x^{\ubar})=0, 
$$
which contradicts the {\em $\G$-liminf inequality} (see Remark~\ref{rem_lirs} in Appendix).

\end{remark}

In this manuscript we focus our attention in presenting dissipativity-based sufficient conditions for both the \emph{equi-coercivity hypothesis} and the \emph{$\Gamma$-convergence hypothesis}, generalizing and developing the results contained in~\cite{DellaFreddi25}, where the relation of these properties with dissipativity theory was not analyzed.

\subsection{Admissible sets defined by controlled ODEs}
\label{sec:ODEandchar}

In this subsection  we introduce the class of differential (controlled) state equations used in the sequel to define the admissible sets $\cA_T$, for $T\in (0,\infty]$. We state the minimal assumption to ensure existence and uniqueness of solutions with respect to control inputs in $\cU=L^p((0,\infty),\R^m)$   with  $p\in(1,\infty]$.


Given $a:[0,\infty)\times \R^n\to \R^n$, $b:[0,\infty)\times \R^n\to \R^{n\times m}$  and $x_0\in \R^n$ we consider the Cauchy problem defined by 
\begin{equation}\label{eq:CauchyProblem}
\begin{cases}
x'(t)= a(t,x(t))+b(t,x(t))u(t),\\
x(0)=x_0.
\end{cases}
\end{equation}
Such class of systems is often referred to as \emph{(nonlinear) systems affine in control}
or simply
\emph{control-affine nonlinear systems}.
We now introduce some basic hypotheses on the functions $a$ and $b$ that ensure existence and uniqueness of solutions  to problem~\eqref{eq:CauchyProblem}. 

\begin{assumption}[Carath\'eodory-Lipschitz conditions]\label{assumpt:Basic0_L1} 
The functions $a$ and $b$ satisfy the following properties:
\begin{enumerate}[leftmargin=*]
\item there exists $M\in L^{1}_{\rm loc}([0,\infty))$ and, for any compact set $H\subset\R^n$, there exists $A_H\in L^{1}_{\rm loc}([0,\infty))$ such that
\[
\begin{aligned}
 &|a(t,0)|\leq M(t)\;\;\mbox{ for a.a.}\  t\ge0,\\
&|a(t,x_1)-a(t,x_2)|\leq A_H(t)|x_1-x_2| \;\;\text{ for all }  x_1, x_2 \in H\mbox{ and a.a.}\ t\ge0;
\end{aligned} 
\]
\item  there exists  $N\in L^{p'}_{\rm loc}([0,\infty))$ and, for any compact set $H\subset\R^n$, there exists $B_H\in L^{p'}_{\rm loc}([0,\infty))$ such that 
\[
\begin{aligned}
 &|b(t,0)|\leq N(t)\;\;\mbox{ for a.a.}\  t\ge0,\\
&|b(t,x_1)-b(t,x_2)|\leq B_H(t)|x_1-x_2| \;\;\text{ for all }  x_1, x_2 \in H\mbox{ and a.a.}\ t\ge0.
\end{aligned} 
\]
\end{enumerate}
\end{assumption}

\noindent Throughout the paper we suppose that Assumption~\ref{assumpt:Basic0_L1} be satisfied.

 Given $u\in \cU$, any function $f_u:[0,\infty)\times \R^n\to \R^n$ such that
 $f_u(t,x):=a(t,x)+b(t,x)u(t)$ for almost all $t\in [0,\infty)$ and all $x\in \R^n$,  
 satisfies the classical Carath\'eodory-Lipschitz conditions for existence and uniqueness of solutions.
This implies that, for any $u\in \cU$ and any $x_0\in 
\R^n$, there exists a \emph{unique locally absolutely continuous maximal solution} (see~\cite[Section I.5]{Hale}) to the Cauchy problem \eqref{eq:CauchyProblem}
denoted by $x^{u,x_0}:\text{dom}(x^{u,x_0})\to \R^n$. 
Here, $\text{dom}(x^{u,x_0})$ denotes the effective  interval of definition of the maximal solution, and it is thus of the form $[0,\tau)$ with $\tau>0$ possibly equal to $\infty$. We recall that we can identify the space of locally absolutely continuous function on $[0,\tau)$ with the Sobolev space $\Wl^{1,1}([0,\tau),\R^n)$.
For a formal definition of such (local) Sobolev space, we refer to~\cite[Appendix]{DellaFreddi25}.\\
Summarizing, under Assumption~\ref{assumpt:Basic0_L1},  given any $u\in \cU$ and any $x_0\in \R^n$ there exists a unique $x\in \Wl^{1,1}([0,\tau),\R^n)$  maximal solution 
to the Cauchy problem~\eqref{eq:CauchyProblem}, denoted by $x^{u,x_0}$.

Whenever the initial condition $x_0\in \R^n$ and/or the control $u$ are fixed or evident from the context, 
we shall simplify the notation by writing $x^{u,x_0}$ as $x^u$, or simply as $x$.

 \subsection{Sequences of control problems with increasing horizon}\label{subsec:SequenceofContProb}

In this section we introduce the considered class of  (sequences of)
optimal control problems.
Let us recall that the control space $\cU=L^p((0,\infty),\R^m)$ is  endowed by  the weak$^\star$ topology. The space of states is, from now on,  considered to be $\cX=\Wl^{1,1}([0,\infty),\R^n)$, and it is equipped with \emph{the strong topology of $L^\infty_{\loc}([0,\infty),\R^n)$}, i.e., the topology of the uniform convergence on compact subintervals of $[0,\infty)$, for reasons illustrated in the subsequent Remark~\ref{rem:Choice}.

 We consider  $a:[0,\infty)\times \R^n\to \R^n$ and $b:[0,\infty)\times \R^n\to \R^{n\times m}$, satisfying Assumption~\ref{assumpt:Basic0_L1}, an initial state $x_0\in \R^n$
and an integrand $\ell:(0,\infty)\times \R^n\times \R^m\to (-\infty,+\infty]$.
The class of problems considered in this manuscript, 
parametrized by the final time $T\in (0,+\infty]$, are of the form
\begin{subequations}\label{eq:OptimalControlProblem}
\begin{equation}\label{eq:CostFunctionalSec3}
\inf_{(u,x)\in \cU\times \cX}J_T(u,x):=\inf_{u\in \cU}\int_0^T \ell(t,x(t),u(t))\,dt
\end{equation}
subject to
\begin{equation}\label{eq:CauchyProblem1}
x'(t)
=a(t,x(t))+b(t,x(t))u(t)\;\;\;\;\;\;\text{in }\; [0,T],
\end{equation}
\begin{equation}\label{eq:Initialcond1}
x(0)=x_0.
\end{equation}
\end{subequations}

 Let us stress the fact that we are assuming, for simplicity, that $x(t)$ be a solution of the Cauchy problem on the closed interval $[0,T]$. In other words, controls $u$ are considered to be admissible only if the maximal interval $[0,\tau)$ of existence of the corresponding solution $x^{u,x_0}$ satisfies $\tau>T$. Such a requirement will allow for a convenient reformulation of the problem presented in the forthcoming Subsection~\ref{Sec:Boundedness} and arise, for instance, if a global solution exists on the space  $\cX:=W^{1,1}_{\loc}([0,\infty),\R^n)$. Considering as ``admissible'' controls leading to unbounded solutions on $[0,T]$ (even if the cost is finite) would lead to more technical difficulties. 

\begin{remark}[The choice of the state space and its topology]\label{rem:Choice}
We have chosen the state space to be the topological space
\begin{equation}\label{defX}
\cX=W^{1,1}_{\loc}([0,\infty),\R^n)
\end{equation}
equipped with the strong topology of $L^\infty_{\loc}([0,\infty),\R^n)$, i.e., the topology of uniform convergence on compact subintervals of $[0,\infty)$.
The main technical motivations for such choice are illustrated in what follows, in a descriptive sense.
In order to be suitable to satisfy the conditions of Theorem~\ref{thm:PatternPreserving}, the topologies selected on $\cU$ and $\cX$ must strike a delicate balance: 
they should be sufficiently strong to ensure $\Gamma$-convergence of the sequence of joint functionals, 
while remaining weak enough to guarantee equi-coercivity of these functionals.  
Under standard growth conditions, weak or weak* topologies naturally emerge as plausible candidates.  
While such topologies often suffice for the control space, the weak topology of $W_{\loc}^{1,1}$ turns out to be ``too weak'' to allow passage to the limit in the state equations.  
Indeed, it is known (see, for example, \cite[Proposition~5.3.3]{Buttazzo89}) that a sufficient condition for passing to the limit is the (local) uniform convergence of the states.  
However, the embedding $W_{\loc}^{1,1}\subset L_{\loc}^\infty$ is not compact, and weak convergence in $W_{\loc}^{1,1}$ does not imply local uniform convergence.  
To address this issue, we directly endow $\cX$ with the 
the strong topology of  $L_{\loc}^\infty([0,\infty),\R^n)$,
as in~\cite{DellaFreddi25}.   
Analogously to the control space $\cU$, convergence in $\cX$ with respect to this topology will simply be denoted by the symbol ``$\to$''.
 \end{remark}
We recall an instrumental lemma that relates different notions of convergence and which is used in what follows.
\begin{lemma}\label{lem_ucbs}
 Let $q\in(1,\infty]$.  The space $\Wl^{1,q}([0,\infty),\R^n)$ is compactly embedded in $L^\infty_\loc([0,\infty),\R^n)$. 
In particular, weakly$^\star$ converging sequences in $\Wl^{1,q}([0,\infty),\R^n)$  are strongly converging in $L^\infty_\loc([0,\infty),\R^n)$, that is, uniformly converging on the compact subsets of $[0,\infty)$. 
\end{lemma}
For the formal definition and a detailed discussion of local Sobolev spaces and the corresponding notions of convergence, we refer the reader to~\cite[Appendix]{DellaFreddi25} and the references therein.

On the running cost function $\ell:[0,\infty)\times \R^n\times \R^m \to \overline \R$ in~\eqref{eq:CostFunctionalSec3} we make the following assumptions. Let us suppose that there exist a  normal integrand $\ell_1:[0,\infty)\times\R^n\to \overline \R$ and a  normal convex integrand
$\ell_2:[0,\infty)\times \R^n\times \R^m \to \overline \R$ such that 
\begin{enumerate}[label=(\alph*)]\setcounter{enumi}{0}
\item \emph{(separation)} $\ell(t,x,u)=\ell_1(t,x)+\ell_2(t,x,u)$ for all $(t,x,u)\in [0,\infty)\times \R^n\times \R^m$;
\item 
\emph{(existence of a  bounded feedback control strategy)} there exists $u_f:(0,\infty)\times \R^n\to \R^m$ Borel-measurable  such that, for any compact set $K\subset \R^n$ 
\begin{enumerate}
\item[(b1)] 
  there exists $\alpha_K\in L^p(0,\infty)$ such that
$$
|u_f(t,x)|\le\alpha_K(t)\quad \forall\, x\in K \mbox{ and a.a.}\ t>0;
$$
\item[(b2)] 
 there exists $m_K\in  L^1(0,\infty)$ such that
\begin{equation}\label{eq:StructureofTheCost1}
\ell_2(t,x,u_f(t,x))\leq m_K(t)\;\;\;\;\forall\,x\in K\text{ and a.a.\ } t>0;
\end{equation} 
\end{enumerate}
\item \emph{(lower bound)} there exists $\gamma_i\in L^1(0,\infty)$ ($i=1,2$) such that 
\begin{equation*}
\begin{cases}
   \ell_1(t,x)\ge -\gamma_1(t),\\  \ell_2(t,x,u)\ge -\gamma_2(t),
   \end{cases}
    \;\;\;\; \forall\,x\in \R^n,\ \forall\,u\in\R^m\mbox{ and a.a.}\ t>0.
\end{equation*}
\end{enumerate}

\begin{remark}\label{rem:ell} Let us make the following remarks on the assumptions.
\begin{enumerate}[leftmargin=*]
\item The functions $\gamma_i$ can always be taken non-negative and we assume that it is so for the rest of the paper. Moreover, let  $\gamma:=\gamma_1+\gamma_2$.
\item The separation assumption (a) allows to incorporate  state constraints of the form $x(t)\in X$, for a prescribed closed subset $X$ of $\R^n$, into the running cost by taking $\ell_1(t,x)=\chi_X(x)$. On the contrary, adding the indicator of $X$ to the integrand $\ell_2$ is prevented by assumption (b2). 
\item The existence of a control strategy $u_f(t,x)$ with the properties required in (b) is very natural and weak. Indeed, the first candidate is the so-called {\em greedy control strategy}, i.e., a function satisfying  
\begin{equation}
\label{eq:greedy}    
u_g(t,x)\in \arg\min_{u\in U}\ell_2(t,x,u),\;\;\;\text{for a.a. $t$ in $(0,\infty)$.}
\end{equation}
The choice $u_f=u_g$ usually works in a wide range of applications  including the very relevant one 
in which $\ell_2$ is of the ``discounted'' form
  \[
  \ell_2(t,x,u) = e^{-\lambda t} r(x,u),
  \]
  where $r:\R^n\times \R^m\to [0,\infty)$ is continuous in $x$ and convex in $u$, and $\lambda > 0$ is a discount parameter. This setting is classical in \emph{economic optimal control}, which  focuses on minimizing performance indices that directly represent operational efficiency. In the discounted-cost setting, future costs are exponentially weighted by a discount factor, emphasizing short-term performance and ensuring convergence of the infinite-horizon cost functional.
\item In the  dissipative framework developed below, hypothesis (b1) suits better compared to the corresponding one made in \cite{DellaFreddi25}. This is clear from the proof of the following $\G$-convergence theorem (Proposition \ref{prop:Gamma-Convergence}).
Moreover, there are simple problems in which the actual version of (b1) is satisfied while the former is not, as the following example shows.
\end{enumerate}
\end{remark}

\begin{example}
Consider the problem to minimize the functional 
$$
J_T(u,x)=\int_0^T|u(t)+x(t)|^2\,dt
$$
under the state equation 
$$
\begin{cases}
x'=u\\
x(0)=x_0>0
\end{cases}
$$
and $\cU=L^2(0,\infty)$. The system is dissipative with respect to the running cost with the coercive storage function $S(x)=x^2$ (see the subsequent Definition~\ref{defn:Dissipativity}). By choosing $\ell_1=0$ and $\ell_2(t,x,u)=|u+x|^2$ we have that the greedy control strategy is $u_g(t,x)=-x$. With the choice $u_f=u_g$ (see \eqref{eq:greedy}), all assumptions (a), (b) and (c) are satisfied. Nevertheless, assumption (b1) of \cite{DellaFreddi25} is not satisfied. 
\end{example}

\subsection{A suitable re-formulation} \label{Sec:Boundedness}
We now re-cast the optimal control problem 
stated   in~\eqref{eq:OptimalControlProblem} in a form more suitable to our purposes  by replacing the admissible set with some other suitable smaller sets $\cA_T$ for $T<\infty$.
To this aim,  we use the feedback strategy $u_f$ introduced in assumption (b).  
Under Assumption~\ref{assumpt:Basic0_L1} we consider problems, parametrized by $T\in (0,\infty]$, defined by
\begin{subequations}\label{eq:RiscritturaProblem}
\begin{equation}\label{eq:AbstractOptContrProb00}
\begin{aligned}
\begin{array}{c}
\displaystyle\inf
\cJ_T(u,x)\\[-1ex]
\ \\[-1ex]
\text{subject to}\\[-1ex]
\ \\[-1ex]
(u,x)\in \cA_T,
\end{array}
\end{aligned}
\end{equation}
with 
\begin{equation}\label{eq:AdmissibleSetA_T}
\begin{aligned}
& \cA_T := \Biggl\{ (u,x)\in \cU\times\cX \;\Big|\;
x^{u,x_0}\mbox{ exists in }[0,T],\text{ and }\\
&\hspace{26ex}   x(t)=\begin{cases}
x^{u,x_0}(t) & \text{if } t\in[0,T],\\
x(T)e^{-(t-T)} & \text{if } t\in(T,+\infty),
\end{cases} 
 \\ &\hspace{26ex}  u(\cdot) \equiv u_f(\cdot,x(\cdot))  \; \text{in } (T,+\infty)\hspace{12ex}\Biggr\},\\ 
&\cA_\infty := \Bigl\{ (u,x)\in \cU\times\cX \;\Big|\;
x = x^{u,x_0} \;\; \text{in } [0,\infty) \Bigr\},
\end{aligned}
\end{equation}
where the notation $x^{u,x_0}$ has been introduced in Subsection~\ref{sec:ODEandchar} to denote the unique maximal solution to \eqref{eq:CauchyProblem1}-~\eqref{eq:Initialcond1}.
It is worth noting that $\cA_T$ is well defined  because the definition of $x$ does not depend on the values of $u$ after time $T$. In other words, the control $u$ can be modified after time $T$ without affecting $x$. Moreover, $x(T)\in\R^n$ because it is required that the solution exists in the closed interval $[0,T]$ (according to the formulation of Problem \eqref{eq:OptimalControlProblem}).    About the modification of the control, since $u_f:(0,\infty)\times \R^n\to \R^m$ is Borel measurable and $x$ is continuous, then the composition  $ u_f(\cdot,x(\cdot)):(0,\infty)\to\R^m$ is Lebesgue measurable. 
Moreover,  by the local boundedness of $u_f$ (hypothesis~\textnormal{(b1)}) and since since $x$ is bounded 
then we have that $u_f(\cdot,x(\cdot))\in L^p((T,\infty),\R^m)$.

The cost functional, for any $T\in (0,\infty]$, denoted by $\cJ_T:\cU\times \cX\to \overline \R$ is defined as
\begin{equation}\label{eq_coerick}
  \cJ_T(u,x):=\int_0^T \ell(t,x(t),u(t))\,dt,
\end{equation}
where, we recall, $\ell:(0,\infty)\times \R^n\times \R^m\to (-\infty,+\infty]$ is an integrand.
For every $T\in (0,\infty]$ we can thus define the \emph{joint functional} $\cF_T:\cU\times\cX\to \overline \R$ by
\begin{equation}\label{eq:JointFunctional}
\cF_T=\cJ_T+\chi_{\cA_T}.
\end{equation}
\end{subequations}
It is worth noting that problems~\eqref{eq:OptimalControlProblem} and~\eqref{eq:RiscritturaProblem} are equivalent from the point of view of the pattern-preserving property. That is, the following proposition holds.

\begin{prop}
 The original family of optimal control problems in~\eqref{eq:OptimalControlProblem} is pattern-preserving if and only if its reformulation in~\eqref{eq:RiscritturaProblem} is also pattern-preserving.    
\end{prop}

\begin{proof}
Let us denote by $\cA^o_T$ the admissible set defined by \eqref{eq:CauchyProblem1} and \eqref{eq:Initialcond1} , 
and $\cF^o_T$ the corresponding joint functional. By definition we have $\cA_T\subseteq\cA_T^o$ and,  
therefore,  $\inf\cF_T\ge\inf\cF_T^o$ for every $T<\infty$ (and the equality holds when $T=\infty$). On the other hand,  the equality  between the infima holds also when $T<\infty$. Indeed,  
if $(u,x)\in\cA_T^o$ then the pair  $(\wt u,\wt x)\in \cU\times \cX$ defined by 
\begin{subequations}\label{eq:TailReplacement}
\begin{equation}
\wt x(t):=\begin{cases}
    x^{u,x_0}(t)&\text{if }t\leq T,\\
 x^{u,x_0}(T)e^{-(t-T)}&\text{if }t> T,
\end{cases}
\end{equation}
\begin{equation}
\wt u(t):=\begin{cases}
u(t)&\text{if }t\leq T,\\
u_f(t,x(t))&\text{if }t> T,
\end{cases}
\end{equation}
\end{subequations}
belongs to $\cA_T$ and the cost in $[0,T]$ is the same.  Hence the claimed equality. 

Let, now, $0<T_k\to\infty$ and suppose that $\cF_{T_k}$ is pattern-preserving. To prove that also $\cF_{T_k}^o$ is pattern-preserving we have to consider a minimizing sequence $(u_k,x_k)$ 
for $\cF_{T_k}^o$ that satisfies conditions  1. and 2.
of Definition \ref{defn:PatternPreserving}. Then, it is easily seen that also the modified sequence  $(\wt u_k,\wt x_k)$ (obtained as in \eqref{eq:TailReplacement} with $T_k$ in place of $T$)  satisfies the same conditions with the same {\em patterns} $N$, $\tau_J^k$,  $u_j^k$ and $u_\infty$. Since, by assumption, $\cF_{T_k}$ is pattern-preserving then $u_\infty$ is optimal for $\cF_\infty$ which, on the other hand, is the same for the original and the reformulated problem. Then we conclude that, by definition, also  $\cF_{T_k}^o$ is pattern preserving. The vice-versa is similar and left to the reader. 
\end{proof}

\begin{remark}[Motivations behind the new formulation]\label{rem_efc}
The new formulation of the admissible set $\cA_T$ in~\eqref{eq:AdmissibleSetA_T} when $T$ is finite has the advantage of providing a way to monitor/bound the behavior of the control $u$ and the state $x$ on the tail interval $[T,+\infty)$, which is not involved in the integral defining the cost at stage $T$.
 This in fact ensures that $u$ is bounded on $[T,\infty)$ (compatible with the condition $u\in \cU$).

Also the state $x$ is forced,  on the unbounded interval $[T,\infty)$, to be decreasing, converging to $0$ and such that $x_{\vert [0,\infty)}\in W^{1,q}([T,\infty),\R^n)$, for all $q\in [1,\infty]$.
Moreover, this behavior on the mentioned tail is designed to preserve continuity of $x$ at time $T$, according  with the requirement that $x\in \cX=\Wl^{1,1}([0,\infty),\R^n)$.

Both these features will be crucial in proving equi-coercivity and $\Gamma$-convergence when considering sequences of problems $\cF_{T_k}$ (defined as in~\eqref{eq:JointFunctional}) for $T_k\to \infty$.
\end{remark}

\section{Dissipativity and pattern-preserving property}\label{sec:ParticularCases}
In this section we introduce the concept of \emph{dissipativity} of control dynamical systems, and we illustrate how it can relate with the \emph{pattern-preserving property}, as stated in Definition~\ref{defn:PatternPreserving}, via  Theorem~\ref{thm:PatternPreserving}.
In particular, we show that dissipativity (that is formally defined below), can be seen as:
\begin{itemize}[leftmargin=*]
\item a crucial sufficient condition that ensures equi-coercivity (recall Definition~\ref{defn:EquiCoercivityGeneral}) for the sequence of optimal control problems defined in~\eqref{eq:RiscritturaProblem}, for any sequence $0\leq T_k\to \infty$;
\item  a convenient assumption in proving $\Gamma$-convergence of the sequences  $\cF_{T_k}$ towards $\cF_\infty$, that is the other main ingredient of our ``abstract'' patter-preserving Theorem~\ref{thm:PatternPreserving}.
\end{itemize}
This will allow us to establish the pattern-preserving property for some families of optimal control problems that have not been explicitly treated  in~\cite{DellaFreddi25}. \\
To this end, we introduce the formal definition of \emph{dissipativity}, which is classical in control theory.

\begin{definition}[Dissipativity]\label{defn:Dissipativity}
Let us consider a functional control space $\cU$.
The state equation~\eqref{eq:CauchyProblem1} is \emph{dissipative} with respect to the \emph{supply function} $r:(0,\infty)\times \R^n\times \R^m\to \overline \R$ if there exists  $S:\R^n \to \R$, 
called a \emph{storage function}, such that
\begin{equation}
S(x)\geq 0,\;\;\;\forall \;x\in \R^n,
\end{equation}
\begin{equation}\label{eq:Dissipativityequation}
\begin{aligned}
    S(x^{u,x_0}(t))-S(x_0)\leq \int_0^t r(s,x^{u,x_0}(s), u(s))\,ds,\;\;\;\; \\ 
\forall u\in \cU,\;\forall x_0\in \R^n, \;\forall \,t\in \mathrm{dom}(x^{u,x_0}).
\end{aligned}
\end{equation}
\end{definition}

In the rest of the paper, we always suppose that the supply rate is an integrand, so that that the integral in the right-hand side of~\eqref{eq:Dissipativityequation} is well-defined (as extended real number).
To our purposes, if we were interested in a known and fixed initial condition $x_0\in \R^n$, we could impose a dissipativity property only for solutions starting at $x_0$. Since this is not common in control theory and does not provide any significant simplification from the technical point of view, we do not pursue this path. Similarly, one could consider a \emph{local} version of Definition~\ref{defn:Dissipativity} in which only particular subsets of $\R^n$ are considered, see for example~\cite{BorMash06}. Even tough this can be natural in some situations (notably, state-constraint optimal control problem), for simplicity we only consider the global notion of dissipativity, as introduced in Definition~\ref{defn:Dissipativity}.\\
We recall  (see for example~\cite[Chapter~3]{VanDerSchaft17})
that, if $S\in \cC^1(\R^n)$, then~\eqref{eq:Dissipativityequation} is equivalent to the condition
\begin{equation}\label{eq:DifferentialCoercivity}
\begin{aligned}
    \inp{\nabla S(x)}{a(t,x)+b(t,x)u}\leq r(t,x,u)\; \; \;\forall (x,u)\in \R^n\times U,\text{ for a.a. }t\in (0,\infty).
\end{aligned}
\end{equation}
\begin{remark}[Brief history of dissipativity and its applications in optimal control]
Dissipativity theory, introduced axiomatically by Jan C. Willems in the early 1970s (see~\cite{Willems73,Willems1972b} and also~\cite{BorMash06} for a recent overview), provides a  unifying framework for several classical notions in control theory, including  stability, passivity and storage-function methods for  dynamical systems. Its central construct (a storage function paired with a supply rate) gives a physically intuitive certificate of how ``supplied energy is stored or dissipated'' by the system. Historically, this viewpoint bridged frequency-domain concepts and state-space Lyapunov analyses, enabling robust and interconnection-based controller design for both linear and nonlinear systems.

In recent years the interplay between dissipativity and (infinite-horizon) optimal control problems has experienced renewed interest. Dissipativity has been shown to underpin \emph{turnpike phenomena}: optimal trajectories for long-horizon problems spend most of the time near a steady optimal operating point, a fact that can be proved using suitable storage-function inequalities. This connection has direct consequences for model-predictive control (MPC) design and numerical strategies for solving large-horizon and infinite-horizon problems (see \cite{Hara2023,Gruene2022, Faulwasser21, Angeli2024} and references therein). For a more direct illustration of how our subsequent results relate to the turnpike-based analysis, we refer the reader to the numerical examples presented in Section~\ref{Sec:Examples}.

Overall, the historical energy-based language introduced by Willems has matured into a practical set of tools that both explain qualitative optimal-control behavior (turnpikes, inverse optimality) and provide constructive constraints for contemporary algorithmic approaches in infinite-horizon formulations of optimal controls problems.
We pursue here this line of research, highlighting the relations between dissipativity and pattern-preserving property (see Definition~\ref{defn:PatternPreserving}) of sequences of optimal control problems with increasing time-horizons.
\end{remark}
In what follows, we say that a function $S:\R^n\to \R$ is \emph{coercive} if the sublevel sets $L_S(a)=\{x\in \R^n\vert\;S(x)\leq a\}$ are bounded, for all $a\in \R$.\\
We now present the main result below.
\begin{thm}[Dissipativity and pattern-preservation]\label{thm:Main}
Let us consider the class of optimal control problems in~\eqref{eq:RiscritturaProblem}.  Besides assumptions (a), (b) and (c) stated in Subsection~\ref{subsec:SequenceofContProb} on the running cost $\ell$,   
suppose that 
 \begin{enumerate}[leftmargin=*]
\item 
there exists $q\in (1,p)$ such that Assumption~\ref{assumpt:Basic0_L1}  holds in a stronger form\footnote{ Note that $\frac{pq}{p-q}>q$ and also $\frac{pq}{p-q}>p'$. In the case $p=\infty$, with an abuse of notation we set $\frac{pq}{p-q}=q$, as in the limit.}   with $M\in L_\loc^q([0,\infty))$, $N\in L_\loc^{\frac{pq}{p-q}}([0,\infty))$ and $A_H\in L_\loc^q([0,\infty))$, $B_H\in L_\loc^{\frac{pq}{p-q}}([0,\infty))$ for any compact set $H\subset \R^n$;
\item  the control system be dissipative w.r.t.\ the running cost $\ell$ with a \emph{coercive} storage function $S:\R^n \to \R$.
\end{enumerate}
Then, for every $x_0\in \R^n$ and any $0\leq T_k\to \infty$, the sequence of functionals $\cF_{T_k}$ (defined in~\eqref{eq:JointFunctional}) is $T_k$-pattern preserving.
\end{thm}

\begin{remark}[Discussions on the hypotheses]
We now discuss the assumptions of the theorem.  
For control-affine systems, Assumption~\ref{assumpt:Basic0_L1}, even in its strengthened form required by the theorem, can be regarded as a mild condition. In particular, it holds trivially for autonomous (a.k.a.\ time-independent)  systems, if the functions $a$ and $b$ are locally Lipschitz continuous with respect to the state variable.

The hypotheses on the running cost function $\ell$ may appear intricate, yet they are in fact quite natural, as already discussed in Remark~\ref{rem:ell}. We want to stress again that the component of the cost function independent of the control (denoted by $\ell_1$) is a  normal integrand and it is allowed to   
take the value $+\infty$. Consequently, this term can also be used to model \emph{state constraints}, by writing
$
\ell_1(t,x) = \widetilde{\ell}_1(t,x) + \chi_X(x)$, where $X \subset \R^n$ denotes a closed (to preserve lower-semicontinuity) state-constraint set and $\widetilde{\ell}_1$ is a (proper) nonnegative normal integrand. 
It is worth observing that, if $X$ is compact and the system is dissipative with respect to the cost, then the coercivity requirement imposed on the storage function $S$ in hypothesis~{2.} of Theorem~\ref{thm:Main} is not essential, and can be avoided. Indeed, given any storage function $S:\R^n \to \R$, one can construct a modified storage function that coincides with $S$ on $X$ and is rendered coercive outside $X$, while preserving the dissipation inequality~\eqref{eq:Dissipativityequation}. Consequently, the decrease condition remains unaffected by the modification.
Similarly, if $X \subset \mathbb{R}^n$ is compact and the running cost $\ell$ is nonnegative, the dissipativity assumption can be entirely dispensed. Indeed, it is sufficient to consider the function $S:\mathbb{R}^n \to \mathbb{R}$ defined by $S(x)=\mathrm{dist}_X(x):=\min_{z\in X}|x-z|$.
One readily verifies that $S$ is a coercive storage function in this setting.

Regarding the control-dependent part of the running cost (i.e., the integrand $\ell_2$), we assume it to be a convex normal integrand, and can be used to model \emph{control  constraints} of the form  \begin{equation}\label{eq:COntrolCOntraint} 
u(t)\in U\ \mbox{for a.a.} \ t\in(0,\infty),
\end{equation}
for a convex and closed set $U\subset \R^m$ by means of an additive contribution of the form  $\chi_{U}(u)$, that is, 
by requiring that $\ell(t,x,u)=+\infty$ whenever  $u\in\R^m\setminus U$.
In this case, the compactness and convexity of the set of control values $U\subset \R^m$  is required by the convexity and lower semicontinuity of $\ell_2$, and   is also a classical assumption.

The assumption (b) on the integrand $\ell$ has already been discussed in Remark~\ref{rem:ell}, revealing that several classes of running costs commonly considered in the literature fall within the scope of our assumptions.

Then, arguably, the most demanding assumption is the \emph{dissipativity} of the state equation with respect to the running cost and with a coercive storage function.  Besides simple cases in which it is trivially satisfied, like the already mentioned ones, this property must therefore be verified on a case-by-case basis,  for instance by searching for suitable storage functions within candidate classes such as quadratic, polynomial, or rational functions, and enforcing the associated dissipation inequalities either via analytical arguments or numerical methods, e.g., linear matrix inequalities (LMIs)~(see e.g.~\cite{BEFB:94}), sum-of-squares (SOS) relaxations~(see e.g.~\cite{sostools,caltechthesis}), or other related tools from systems theory. These are precisely the same families of computational techniques that are routinely employed in Lyapunov theory for linear, nonlinear, and hybrid systems. 
Summarizing, dissipativity verification closely parallels classical Lyapunov analysis. See, for instance,~\cite[Chapter~1]{Arcak2016} for a review of numerical methods for dissipativity analysis.
From this perspective, the theorem can, rather bluntly, be summarized by the catchy statement that ``\emph{dissipativity implies the pattern-preserving property}''. Even if this is a tempting and fallacious shortcut, the main message of this manuscript is precisely to highlight the deep connection between these two seemingly distant notions.
\end{remark}
The next subsections are devoted to prove Theorem~\ref{thm:Main}. We split the presentation into the two main hypotheses of our abstract pattern-preserving sufficient criterion provided by Theorem~\ref{thm:PatternPreserving}: \emph{equi-coercivity (w.r.t. $x$)} and \emph{$\Gamma$-convergence}. 

\subsection{Dissipativity and equi-coercivity}
In this subsection we provide our first intermediate result in proving~Theorem~\ref{thm:Main}, relating the concept of dissipativity with equi-coercivity.
\begin{prop}[Sufficient conditions for equi-coercivity]\label{prop:EquiCoercivity}
Let us consider the class of optimal control problems in~\eqref{eq:RiscritturaProblem}. Under the assumptions of Theorem \ref{thm:Main}, for any initial condition $x_0\in \R^n$ and any sequence $0\leq T_k\to \infty$, we have
\begin{enumerate} \item\label{it:boundbound}  for any 
$C>0$ and any sequence $(u_k,x_k)$ such that $u_k\to u$ in $\cU$ and $\cF_{T_k}(u_k,x_k)\le C$ $\forall\,k\in\N$, the sequence $x_k$ is bounded in $L^\infty([0,\infty),\R^n)$;
\item\label{it:diss-coerc} the sequence of functionals $\cF_{T_k}$ (defined in~\eqref{eq:JointFunctional}) is sequentially equi-coercive w.r.t. $x$.
\end{enumerate}
\end{prop}

\begin{proof}
Let us fix any $x_0\in \R^n$ and any sequence $0\leq T_k\to \infty$.  Consider $C>0$ and a sequence $(u_k,x_k)$ such that $u_k\to u$ in $\cU$ and $\cF_{T_k}(u_k,x_k)\leq C$.
By definition of $\cF_{T_k}$ this first implies that $
(u_k,x_k)\in \cA_{T_k}$ for all $k\in \N$, i.e., 
\[
x_k(t)=\begin{cases}x^{u_k,x_0}(t) \;\;\;&\text{if } t\in [0,T_k]\\
x_k(T_k)e^{-(t-T_k)}\;\;\;&\text{if } t\in (T_k,+\infty),\end{cases}
\]
and
\[
u_k(t)=u_f(t,x_k(t)) 
\text{ a.e.\ in }(T_k,\infty).
\]

By the dissipativity hypothesis and the non-negativity of $\gamma$ and $\ell+\gamma$, we have that
 \[
 \begin{aligned}
S(x_k(t))&-S(x_0)\leq  \int_0^t\ell(s,x_k(s),u_k(s))\,ds \leq \int_0^t\ell(s,x_k(s),u_k(s))+ |\gamma(s)|\;ds
\\&\leq \int_0^{T_k}\ell(t,x_k(t),u_k(t))+ |\gamma(t)|\;dt\leq \int_0^{T_k}\ell(t,x_k(t),u_k(t))\,dt+\int_0^{T_k} |\gamma(t)|\;dt\\& \leq C+\|\gamma\|_1,
\end{aligned}
 \]
 for all $k\in \N$ for all $t\in [0,T_k]$.
 Since the storage function $S$ is coercive, we have that there exists a compact set $K\subset \R^n$ such that 
 \[
x_k(t)\in K\;\;\;\;\forall k\in \N,\;\;\forall \;t\in [0,T_k].
 \]
 Since on the tails $[T_k,\infty)$ the functions $x_k$ are all smaller (in norm) than $x_k(T_k)$, 
 we have that
 \[
x_k(t)\in \text{conv}\{K,-K\}=:H\;\;\;\;\forall k\in \N,\;\;\forall \;t\ge0,
 \]
which also implies that  there exists a $M_1>0$ such that  $\|x_k\|_\infty\leq  M_1$, 
proving Item~\ref{it:boundbound}.\\
Now, using Assumption~\ref{assumpt:Basic0_L1}  we have
\begin{equation}\label{eq:DerivativeProveConvergence}
\begin{split}
    &\left|x'_k(t)\right|  \leq |a(t,x_k(t))|+|b(t,x_k(t))||u_k(t)| \\ 
    &\ \leq M(t)+A_{H}(t)M_1+(N(t)+B_H(t)M_1)|u_k(t)|.
\end{split}
\end{equation}
By using Holder's inequality (multiple times) and by boundedness of $u_k$ in $\cU$ (since it is a converging sequence), our standing assumptions on the functions $A_{H},B_H, M$ and $N$ imply that $x_k'$ is bounded in $L^q_{\loc}([0,\infty),\R^n)$.
We have thus shown  that the sequence $x_k$ is bounded in $\Wl^{1,q}([0,\infty),\R^n)$, and it thus admits a weakly$^\star$ converging subsequence $x_k\weaks x$ in $\Wl^{1,q}([0,\infty),\R^n)$. Then, by  Lemma~\ref{lem_ucbs}, this implies that $x_k\to x$ in $\cX$. Hence the claimed equi-coercivity w.r.t. $x$.
\end{proof}

\begin{example}
We show that dissipativity is indeed an essential sufficient condition, by presenting an example in which all  other assumptions of Proposition~\ref{prop:EquiCoercivity} hold, but the sequence of optimal control problems is not equi-coercive (w.r.t. $x$).
Let us consider the scalar system defined by
\begin{equation}
\begin{cases}\label{eq:CounterexampleSystem}
x'(t)=\mathbf{1}_{[0,1]}(t)u(t)x^2(t)\\
x(0)=1,
\end{cases}
\end{equation}
with  $u\in \cU=L^\infty(0,\infty)$. The cost, for any $T\in (0,\infty]$, is defined by
\[
\cJ_T(u,x):=\int_0^T e^{-t}u(t)+\chi_{[0,1]}(u(t))\;dt.
\]
It is easily checked that the hypotheses of Proposition~\ref{prop:EquiCoercivity} are satisfied  (with $\ell(t,x,u)=e^{-t}|u|+\chi_{[0,1]}(u)$,  for $p=\infty$ and any $q>1$), with the sole exception of the assumption of dissipativity with a coercive storage function. 
Given $1\le T_k\to \infty$, let us consider the sequence $u_k$ defined by
\[
u_k(t)=\begin{cases}
 1-\frac{1}{T_k}\;\;&\text{if }t\in [0,1],\\
0\;\;&\text{if }t\in (1, \infty).
\end{cases}
\]
 for any $k\in \N$. A solution $x_k\in \Wl^{1,1}([0,\infty))$ to the system~\eqref{eq:CounterexampleSystem} exists and, with the prolongation introduced in~\eqref{eq:AdmissibleSetA_T}, reads
\[
x_k(t)=x^{u_k,1}(t)=\begin{cases}
\frac{1}{1-(1-T^{-1}_k)t}\;\;\;&\text{if }t\in [0,1),\\
T_k\;\;\;&\text{if }t\in [1,T_k),\\
T_ke^{-(t-T_k)}&\text{if }t\in [T_k,\infty).
\end{cases}
\]
Then, we can note that it holds that
\[
\cJ_{T_k}(x_k,u_k)=\int_0^{T_k}e^{-t}u_k(t)\;dt\leq \int_0^\infty e^{-t}\;dt=1.
\]
The sequence $u_k$ is clearly bounded in $L^\infty$, and it holds that $u_k\stackrel{*}\wto u$, with $u=\mathbf{1}_{[0,1]}$. On the other hand, $x_k$ is unbounded in $L^\infty(0,1)$, and thus cannot converge uniformly on compact intervals, violating equi-coercivity for the considered sequence of optimal control problems.
We can also note that the control $u=\mathbf{1}_{[0,1]}$ leads to the solution to~\eqref{eq:CounterexampleSystem} given by
\[
x(t)=\frac{1}{1-t},\;\;\;\;t\in [0,1)
\]
which does not admit any prolongation to a function in $\Wl^{1,1}([0,\infty))$.\\
 It is also worth noting that the system is dissipative in the general definition, i.e., if  non-coercive storage functions are allowed. Indeed, consider, for instance, the candidate storage $S(x) = \frac{e^{-1}}{1+x^2}$. Computing the derivative of $S$, we find that condition~\eqref{eq:DifferentialCoercivity} then reads
\begin{equation*}
    \begin{split}
        \frac{-2x^3e^{-1}}{\left(1+x^2\right)^2} \mathbf{1}_{[0,1]}(t)u \leq e^{-t} u\;\;\;  \;
        \forall (x,u)\in \R^n\times [0,1],\text{ and for a.a. }t\in \R_+.
    \end{split}
\end{equation*}
If $u=0$ or $t>1$, then the inequality is trivially satisfied. If $u>0$, we can divide by $u$ in both sides. Thus, for $t \in [0,1]$ we need to verify
\begin{equation*}
    \frac{-2x^3e^{-1}}{\left(1+x^2\right)^2}  \leq e^{-t} \;\;\;\forall x\in \R^n\text{ and for }t\in [0,1].
\end{equation*}
Since $e^{-t} \geq e^{-1}$ holds for $t \in [0,1]$, it suffices to verify that ${ \frac{-2x^3}{\left(1+x^2\right)^2} } \leq 1$ for all $x\in \R^n$. For $x \geq 0$, this last inequality holds trivially. For $x<0$, the maximum on the left hand side is attained at $x = -\sqrt{3}$ and is given by $\frac{3 \sqrt{3}}{8} < 1$. Therefore, the dissipativity property is  satisfied and $S(x) = \frac{e^{-1}}{1+x^2}$ is a \emph{non-coercive} storage function.

Concluding the (counter-) example, let us note that dissipativity w.r.t.\ the function $\ell(t,x,u):=e^{-t}u$ with a coercive storage function $S:\R^n\to \R$
would imply the following notion of \emph{uniform boundedness}: for any $x_0\in \R^n$ there exists a bounded set $X_{x_0}\subset \R^n$ such that $x^{u,x_0}(t)\in X_{x_0}$ for every $u\in \wt\cU:= L^\infty((0,\infty),[0,1])$ (i.e., the set of controls satisfying the control constraint $u(t)\in [0,1]$) and every $t\in \textit{\rm dom}(x^{u,x_0})$.  Indeed, for any $x_0\in \R^n$ and any $u\in \wt \cU$  we would have
\[
S(x^{u,x_0}(t))\leq \int_0^te^{-s}u(s)\;ds+S(x_0)\leq 1+S(x_0) 
\]
and thus, by coercivity of $S$ there exists a bounded set $X_{x_0}:=\{x\in \R^n\;\vert\;S(x)\leq 1+S(x_0)\}$ such that 
$
x^{u,x_0}(t)\in X_{x_0}$, for all $u\in \wt \cU$ and all $t\in [0,\infty)$.

The fact that uniform boundedness implies equi-coercivity was already noted in~\cite[Subsection 6.1]{DellaFreddi25}. Indeed, regardless of the  running cost function, uniform boundedness is a sufficient condition for the implication
\[
u_k\to u\;\;\Rightarrow\;\;x^{u_k,x_0}\to x^{u,x_0},
\] 
 when $\{u_k\}\subset L^\infty((0,\infty), U)$ with $U$ a compact convex set (see~\cite[Proposition 6,3]{DellaFreddi25}). The \emph{dissipativity property}  (with coercive storage functions) can thus be considered as a running-cost-dependent generalization of uniform boundedness  in order to ensure equi-coercivity. 
\end{example}


 \subsection{Dissipativity and \texorpdfstring{$\Gamma$}{Gamma}-Convergence}

In this section we study the relation between dissipativity and the second main ingredient of proof of Theorem~\ref{thm:Main} (via the abstract Theorem~\ref{thm:PatternPreserving}): the $\Gamma$-convergence 
of the problems $\cF_T$ (defined in~\eqref{eq:RiscritturaProblem}) to $\cF_\infty$, as $T\to \infty$.

As a preliminary step, we  state an instrumental convergence result for solutions to the Cauchy problem~\eqref{eq:CauchyProblem}, which holds only under Assumption~\ref{assumpt:Basic0_L1}, and it is thus independent of the dissipativity assumption.
\begin{lemma}\label{lemma:GammaCOnvergenceofSolutions}
Under Assumption~\ref{assumpt:Basic0_L1}, for any sequence $0\le T_k \to \infty$  and any $x_0 \in \mathbb{R}^n$, if $(u_k, x_k) \to (u, x)$ and $(u_k, x_k) \in \mathcal{A}_{T_k}$ for infinitely many $k \in \mathbb{N}$, then $(u, x) \in \mathcal{A}_\infty$.

\end{lemma}
\begin{proof}
Let us consider $0\leq T_k\to \infty$ and  $(u_k,x_k)\to (u,x)$ such that $(u_k,x_k)\in \AI_{T_k}$ for infinitely many $k\in\N$. By definition of $\AI_{T_k}$, this means that there exists a subsequence (not relabeled) such that   $x_k(t)=x^{u_k,x_0}(t)$ for all $k\in \N$ and all $t\in [0,T_k]$.
Moreover, by hypothesis the sequence $x_k$ is converging to $x\in W^{1,1}_{\loc}([0,\infty),\R^n)$ strongly in $L^\infty_{\loc}([0,\infty),
\R^n)$, i.e., uniformly on compact sets. 
To prove the claim we have to show that  $x=x^{u,x_0}$. This can be done by showing that our assumptions  allow to pass to the limit in the state equation in the weak sense of distributions, that is
\begin{equation*}
\int_0^\infty \hspace{-0.25cm}\varphi(t)\cdot  x'(t)\,dt=\int_0^\infty \varphi(t)\cdot\Big( a(t,x(t)) +b(t,x(t))u(t)\Big)\,dt
\end{equation*}
 for any test function $\varphi\in {\mathcal D}((0,\infty),\R^n)$.
Nevertheless, such equality  can be easily proved by arguing as in the proof  of~\cite[{\em1.}\ of Lemma 5.3]{DellaFreddi25}, and we omit the details. 
Then, $x$ satisfies the limit state equation. Since  $x_k\to x$ uniformly on compact sets and $x_k(0)=x_0$ for all $k\in \N$ then $x(0)=x_0$. Hence $x$ is a solution of the limit Cauchy problem~\eqref{eq:CauchyProblem} that, on the other hand,  admits a unique solution.  Thus we have $x=x^{u,x_0}$ in $[0,\infty)$, proving that $(u,x)\in \cA_\infty$.
\end{proof}

We are now in position to state the main result of this subsection. It is worth noting that, usually, $\Gamma$-convergence results are established independently of compactness properties (like coercivity), and are then combined to guarantee the convergence of minimizers. In our case, however, the dissipativity assumption with a coercive storage function induces a compactness property of the states, which also plays a role in  establishing the liminf inequality in the proof of $\Gamma$-convergence. This constitutes a rather substantial difference compared to the more classical setting of \cite{DellaFreddi25}.

\begin{prop}\label{prop:Gamma-Convergence}
Let us consider the class of optimal control problems in~\eqref{eq:RiscritturaProblem}. Under the assumptions of Theorem \ref{thm:Main}, for any initial condition $x_0\in \R^n$ and any sequence $0\leq T_k\to \infty$, it holds that 
\begin{equation}
\label{eq:GammaF_T}
\cF_\infty=\Gamma^-_{\rm seq}(\cU\times \cX)\lim_{k\to \infty}\cF_{T_k}.
\end{equation}
\end{prop}

\begin{proof}
Let us consider a sequence $0\le T_k\to \infty$. First of all, let us observe that it is not restrictive to assume that the integrands are non-negative. Indeed, it is immediately seen that the function $S$ is a coercive storage function also for the non-negative supply function 
$$
\widetilde\ell(t,x,u)=\ell(t,x,u)+\gamma(t).
$$ 
Moreover, by denoting with $\widetilde\cF_T$ the corresponding joint functionals with $\widetilde\ell$ in place of $\ell$,
we clearly have 
$$
\cF_T(u,x)=\widetilde\cF_T(u,x)-\int_0^T\gamma(t)\,dt
$$
and $$\lim_{k\to\infty}\int_0^{T_k}\gamma(t)\,dt=
\int_0^\infty\gamma(t)\,dt.
$$
Thus, showing that
\[
\G_{\rm seq}^-(\cU\times\cX)\lim_{h\to\infty}\widetilde\cF_{T_h}=\widetilde\cF_\infty.
\]
will immediately imply \eqref{eq:GammaF_T}. 
So, restricting without loss of generality to the case $\gamma=0$ (non-negative integrands), by Remark~\ref{rem_lirs} in Appendix, we have to prove the following:
\begin{enumerate}
\item (\emph{liminf inequality}) for all sequences $(u_k,x_k)\to (u,x)$ we have $\dis \cF_{\infty}(u,x)\leq \liminf_{k\to \infty} \cF_{T_k}(u_k,x_k)$;
\item (\emph{recovery sequence}) for every $(u,x)\in \cU\times \cX$ there exists a sequence $(u_k,x_k)\to (u,x)$ such that  $\dis \cF_\infty(u,x)\geq  \limsup_{k\to \infty} \cF_{T_k}(u_k,x_k)$.
\end{enumerate}

Let us start by proving the \emph{liminf} inequality.
Let us consider $(u_k,x_k)\to (u,x)\in \cU\times \cX$
and  suppose that 
\[
\liminf_{k\to \infty}\cF_{T_k}(u_k,x_k)<\infty,
\]
otherwise there is nothing to prove.
Thus, possibly passing to a subsequence, we have that there exists $C\geq 0$ such that 
$\cF_{T_k}(u_k,x_k)\leq C$. In particular, since $\cF_{T_k}=\cJ_{T_k}+\chi_{\cA_{T_k}}$ we have that 
\[
(u_k,x_k)\in \cA_{T_k}\;\;\;\;\forall \;k\in \N.
\]
By  our assumptions,~Lemma~\ref{lemma:GammaCOnvergenceofSolutions}
implies that $(u,x)\in \cA_\infty$. Summarizing, we have 
\begin{equation}\label{eq:IdicatrixOk}
\chi_{\cA_\infty}(u,x)=\chi_{\cA_{T_k}}(u_k,x_k)=0,\;\;\;\;\forall \;k\in \N.
\end{equation}
Let us now focus on the  cost functional. Let us recall that we have $\ell(t,x,u)=\ell_1(t,x)+\ell_2(t,x,u)$, where $\ell_1$ is a non-negative normal integrand, while $\ell_2$ is a non-negative normal convex integrand. Thus, for any $T\in (0,\infty]$ we define the functionals  $\cH_T,\cI_T:\cU\times\cX\to [0,+\infty]$ by \vspace{-2ex} 
\begin{equation*}
\hspace{-10ex}\begin{split}
    \cH_T(u,x):=\int_0^T \ell_1(t,x(t))\,dt\;\;\;\; \text{ and }\;\;
\;\; \cI_T(u,x):=\int_0^T\ell_2(t,x(t),u(t))\,dt.
\end{split}
\end{equation*}
By Fatou's Lemma (since $\ell_1$ is nonnegative) we have
\begin{equation*}
    \begin{split}
        \liminf_{k\to \infty}\cH_{T_k}(u_k,x_k)&=\liminf_{k\to \infty}\int_{0}^\infty \ell_1(t,x_k(t))\mathbf{1}_{[0,T_k]}(t)\,dt \\ &\geq \int_{0}^\infty \liminf_{k\to\infty}\ell_1(t,x_k(t))\mathbf{1}_{[0,T_k]}(t)\,dt.
    \end{split}
\end{equation*}
Since $x_k\to x$ uniformly on compact intervals (thus pointwisely) and $\ell_1(t,\cdot)$ is lower semicontinuous, we have 
\[
\liminf_{k\to\infty}\ell_1(t,x_k(t))\mathbf{1}_{[0,T_k]}(t)\geq \ell_1(t,x(t))\text{ for a.a.}\ t\ge0.
\]
Thus, summarizing, we have proved that
\begin{equation}\label{eq:LimInfFUncFirstPart}
\liminf_{k\to \infty}\cH_{T_k}(u_k,x_k)\geq \int_0^\infty \ell_1(t,x(t))\,dt=\cH_\infty(u,x).
\end{equation}
Let us now consider the sequence $\cI_{T_h}$. 
By De Giorgi and Ioffe's Semicontinuity Theorem (see for instance \cite[Theorem 7.5]{FL07} or \cite[Section 2.3]{Buttazzo89}), for 
every $T\in(0,+\infty)$ the functional 
$\cI_T:  \cU\times\cX\to [0,+\infty]$ 
is  lower semicontinuous with respect to the chosen topologies, because  $\ell_2$ is a normal convex integrand. Moreover, also $ \cI_\infty$ is  lower semicontinuous. Indeed, since the integrand is non-negative we have 
\begin{equation*}
\int_0^{\infty}\ell_2\big(t, x(t),u(t)\big)\,dt=\sup_{T>0}\Big(\int_0^{T}\ell_2\big(t,x(t),u(t)\big)\,dt\Big),
\end{equation*}
 implying that $\cI_\infty$ is lower semicontinuous, as the supremum  of any collection of lower semicontinuous functionals is lower semicontinuous. 
 
Since  we have $\cJ_{T_k}(u_k,x_k)\leq C$ for all $k\in \N$,  by Proposition \ref{prop:EquiCoercivity} we also have that there exists a compact set $K$ such that $x_k(t)\in K$ for all $k\in \N$ and all $t\in [0,\infty)$. 
By the structure of the cost functional we thus have that 
\begin{equation*}
    \begin{split}
        0&\leq \liminf_{k\to \infty}\int_{T_k}^\infty \ell_2(t,x_k(t),u_k(t))\;dt \leq \limsup_{k\to \infty}\int_{T_k}^\infty \ell_2(t,x_k(t),u_g(t,x_k(t))\,dt\\ &\leq \lim_{k\to \infty}\int_{T_k}^\infty m_K(t)\;dt=0
    \end{split}
\end{equation*}
since $m_K\in L^1((0,\infty))$ (recall hypothesis (b) in Subsection~\ref{subsec:SequenceofContProb}). 
We have thus proved 
\begin{equation}\label{eq:CodeVanishing}
\lim_{k\to \infty}\int_{T_k}^\infty \ell_2(t,x_k(t),u_k(t))\;dt=0.
\end{equation}
By equation~\eqref{eq:CodeVanishing}  we have
\begin{equation*}\label{eq:LiminfCOsts}
\begin{aligned}
\liminf_{k\to \infty}\cI_{T_k}(u_k,x_k)&=\liminf_{k\to \infty}\int_0^{T_k}\ell_2(t,x_k(t),u_k(t))\,dt  +\lim_{k\to \infty}\int_{T_k}^\infty \ell_2(t,x_k(t),u_k(t))\;dt\\&=\liminf_{k\to \infty}\int_0^{\infty}\ell_2(t,x_k(t),u_k(t)) =\liminf_{k\to \infty}\cI_\infty(u_k,x_k)\geq \cI_\infty(u,x)
\end{aligned}
\end{equation*}
where the last inequality holds by lower semicontinuity of $\cI_\infty$.

This, together with~\eqref{eq:IdicatrixOk}-\eqref{eq:LimInfFUncFirstPart}, concludes the proof of the liminf inequality.

We now prove the existence of a recovery sequence, i.e., given $(u,x)\in \cU\times \cX$ we have to show that there exists a sequence $(u_k,x_k)\to (u,x)$ in $\cU\times\cX$ such that  $\dis \cF_\infty(u,x)\ge \limsup_{k\to \infty} \cF_k(u_k,x_k)$.   We can suppose that $\cF_\infty(u,x)<+\infty$ (otherwise there is nothing to prove). This implies that $(u,x)\in \cA_\infty$ i.e., $x=x^{u,x_0}$. For any $k\in \N$ let us define $(u_k,x_k)\in \cU\times \cX$ by
\[
x_k(t)=\begin{cases}x^{u,x_0}(t) \;\;\;&\text{if } t\in [0,T_k]\\
x(T_k)e^{-(t-T_k)}\;\;\;&\text{if } t\in (T_k,+\infty),\end{cases}
\]
and
\[
u_k(t)=\begin{cases}
u(t) \;\;\;&\text{if }t<T_k\\
u_f(t,x_k(t))\;\;\;&\text{if }t> T_k.
\end{cases}\
\]
This implies that $(u_k,x_k)\in \cA_{T_k}$ for all $k\in \N$ and moreover, that $(u_k,x_k)\to (u,x)$ in the considered topologies. Let us sketch the proof of such convergence. First of all we observe that the sequence $(x_k)$ is bounded in $\cX$, since every seminorm (i.e., the uniform norm on every compact subset of $[0,\infty)$) is bounded by the corresponding one of $x$. Also $(u_k)$ is bounded in $L^p(0,\infty)$, thanks to assumption (b1).
Then, every subsequence of $(u_k,x_k)$ admits a weakly converging further subsequence which can be shown to converge always to $(u,x)$ in the sense of distribution and, hence, in $\cU\times \cX$. Since the topology of the latter is metrizable on bounded sets, then we can conclude that the whole sequence converges as claimed. 

 Since $\ell$ is nonnegative  and $(u_k,x_k)=(u,x)$ on $(0,T_k)$, we also have
\[
\begin{aligned}
\limsup_{k\to \infty}\cF_{T_k}(u_k,x_k)&=\limsup_{k\to \infty}\int_0^{T_k}\ell(t,x_k(t),u_k(t))\;dt  =\limsup_{k\to \infty}\int_0^{T_k}\ell(t,x(t),u(t))\;dt\\&\leq \int_0^{\infty}\ell(t,x(t),u(t))\;dt=\cF_\infty(u,x).
\end{aligned}
\]
We have thus proved that $(u_k,x_k)$ is a recovery sequence, concluding the proof of $\displaystyle
\cF_\infty=\Gamma^{-}_{\rm seq}(\cU\times \cX)\lim_{k\to \infty}\cF_{T_k}$. 
\end{proof}

Finally, we can prove our main result Theorem~\ref{thm:Main}.

\begin{proof}[Proof of Theorem~\ref{thm:Main}]
The proof  follows by  Proposition~\ref{prop:EquiCoercivity} and Proposition~\ref{prop:Gamma-Convergence}, and by  applying  
the abstract Theorem~\ref{thm:PatternPreserving}  to conclude. 
\end{proof}

\section{Examples and Applications}\label{Sec:Examples}
In this section,   we present two elementary  examples illustrating our main results. In the first  we consider an optimal control problem (with  increasing horizon) for a dissipative state equation, that allows us to establish the pattern-preserving property and subsequently deduce the structure of the infinite-horizon optimal control (from the finite-horizon ones). In the second example, we examine a non-dissipative system (more precisely, with no coercive storage function) in which the pattern-preserving property still holds, thereby demonstrating that the sufficient conditions we proposed cannot be regarded as necessary. In Subsection~\ref{subsec:LQR} we explicitly state the consequence of Theorem~\ref{thm:Main} for a class of non-linear quadratic regulator optimal control problems.

\begin{example}[\emph{Dissipative system inducing a pattern-preserving problem}]\label{example:Dissipative1}
Given the final time $T\in (0,\infty]$,
we consider the (sequence of) optimal control problems on $\cU\times \cX$ with $\cU= L^\infty(0,\infty)$ and $\cX=\Wl^{1,1}([0,\infty))$ defined by
\begin{align*}
& J_T(u) = \int_0^T 4u(t)+|x(t)|+\chi_U(u(t)) \, dt, \\
\text{subject to}& \quad  x'(t) = (1-u(t)) x(t), \\ 
&\quad x(0)=1,
\end{align*}
with $U=[0,2]$.
 This problem  satisfies the 
assumptions of Theorem~\ref{thm:Main}, with $\ell_1(x):=|x|$, $\ell_2(u):=4u+\chi_U(u)$ and the feedback control strategy given by $u_f(t,x)=0$ for all $(t,x)\in (0,\infty)\times \R$. It remains to show that the system is dissipative w.r.t to $\ell(u,x)=4u+|x|+\chi_U(u)$ with a coercive storage function.
Let us define $S:\R\to \R$ by $S(x)=\ln(|x|+1)$ with is clearly non-negative and coercive.  
It can easily proved, simply verifying the conditions in Definition~\ref{defn:Dissipativity} (or the differential condition in~\eqref{eq:DifferentialCoercivity}), that $S$ is a storage function w.r.t. $\ell$ for the system under consideration. Then, all the hypothesis of Theorem~\ref{thm:Main} are satisfied  with $p=\infty$: the pattern preserving property is thus satisfied.\\
We now conclude the example by completely characterizing the optimal controls for the finite-horizon (via Pontryagin principle) and the infinite-horizon one (using the pattern-preserving property). 
First of all we note that, for any $u\in \cU$, the corresponding solution $x^{u,1}$
 is defined by
 \[
 x^{u,1}(t)=e^{\int_0^t (1-u(s))\;ds}
 \]
 and it is, thus, strictly positive.
We can thus consider the problems with the simplified running integral cost defined by $\wt \ell(x,u):=4u+x+\chi_U(u)$.
The Hamiltonian function $H:\R^4\to \R$ then reads
\[
\begin{aligned}
H(t,u,x,p)=(4u+x)+p(1-u)x+\chi_U(u)=x+px+u(4-px)+\chi_U(u).
\end{aligned}
\]
Let us fix $T>0$, and let us consider a triple $(u,x,p):[0,T]\to \R^3$ satisfying the Pontryagin principle conditions; in this case, $(u,x)$ is also known as a \emph{Pontryagin extremal}.
The costate $p:[0,T]\to \R$ satisfies
\[
\begin{cases}
p'(t)=-1-(1-u(t))p(t)\\
p(T)=0.
\end{cases}
\]
By the Weierstrass minimum condition, and introducing  the \emph{switching function} $\phi:[0,T]\to \R$ defined as
$
\phi(t):=4-p(x)x(t)$, we have that
\[
u(t)=\begin{cases}
2\;\;&\text{if }\phi(t)<0\\
0\;\;&\text{if }\phi(t)>0.\\
\end{cases}
\]
First of all, since $p(T)=0$, we have $\phi(T)=4>0$ and thus $u(T)=0$. Then, computing the derivative we obtain that
\begin{equation*}
    \begin{aligned}
        \phi'(t)=-x'(t)p(t)-p'(t)x(t) =-(1-u(t))x(t)p(t)+1+(1-u(t))p(t)x(t)=1.
    \end{aligned}
\end{equation*}
Then, the switching function is given by
$\phi(t)=t+(4-T)$
and then the unique switching point (if any) is given by the solution to equation
$
t+(4-T)=0
$, that is $t_s(T):=T-4$. Thus, the unique optimal control $u_T:[0,T]\to [0,2]$ for the problem on the time-horizon $[0,T]$ is (in an almost everywhere sense) defined  by
\[
u_T(t)=\begin{cases}
2\;\;&\text{if }t\in (0,T-4),\\
0\;\;&\text{if }t\in (T-4,T).
    \end{cases}
\]
Since $t_s(T)=T-4\to \infty$ as $T\to \infty$ implies that $u_T\to u_\infty\equiv 2$ (in the weak* sense in $L^\infty((0,\infty),[0,2]$), we can immediately conclude, via Theorem~\ref{thm:Main} that $u_\infty\equiv 2$ is an optimal control for the infinite-horizon problem.\hfill $\triangle$ 
\end{example}
In the following example we slightly modify the previous one, inducing a non-dissipative yet pattern-preserving sequence of problems.

\begin{example}[\emph{Non-dissipative but pattern-preserving problem}]\label{ex:NONDISSIpative}
Given the final time $T\in (0,\infty]$, the control and state spaces $\cU= L^\infty(0,\infty)$, $\cX=\Wl^{1,1}([0,\infty))$
we consider the (sequence of) optimal control problems defined by
\begin{align}
\min_{u \in \cU}   J_T(u) = &\int_0^T \left(\frac{|u(t)|}{3} + |x(t)|\right) e^{-2t}+\chi_U(u(t)) \, dt, \\
\text{subject to } \quad & x'(t) = (1-u(t)) x(t), \\ 
&x(0)=1,
\end{align}
with $U=[0,2]$.
Let us start by noting that the system cannot be dissipative w.r.t.\ the running cost integrand with a coercive function $S:\R\to \R$. Indeed,  otherwise,
by applying the constant control $u\equiv 0$ that  generates the solution  $x(t)= e^t$, we would have
\[
S(e^t)-S(1)\leq \int_0^t(e^s)e^{-2s}\,ds\leq\int_0^\infty e^{-s}\,ds=1
\]
for all $t\in [0,\infty)$, contradicting the coercivity of $S$.
On the other hand it can be formally seen that the corresponding family of optimal control problems is $T_k$-pattern preserving, for any $T_k\to \infty$. This follows from~\cite[Theorem 3.1] {DellaFreddi25}, since the equi-coercivity of the sequence  (recall Definition~\ref{defn:EquiCoercivityGeneral}) can be proved directly, without relaying on dissipativity, at the cost of considerably more tedious technical arguments. We sketch here the argument, without presenting all details for the sake of conciseness. First of all it is easy to see that any admissible control $u$ produces a nonnegative cost  bounded from above by
$\int_0^\infty (\frac{2}{3}+e^t)e^{-2t}dt=\frac{4}{3}$.
For any sequence $u_k\to u$ in $\cU$ with $\chi_U(u_k(t))=0$ for almost all $t\in (0,\infty)$ and for all $k\in \N$, the corresponding solutions 
\[
x^{u_k}(t)=e^{\int_0^t (1-u_k(s))\,ds}\;\;\;\;\forall t\in (0,\infty)
\]
converge, uniformly on compact intervals, to $x^{u}$. This can be achieved by applying the Arzelà–Ascoli theorem on any arbitrary compact interval 
$[0,T]$, since we have:
\[
\begin{aligned}
e^{-t}\leq x^{u_k}(t)\leq e^t \;\;\forall\,k\in \N,\;\forall\,t\in [0,\infty)\;\text{ and }\;
|(x^{u_k})'(t)|\leq e^t \;\;\forall\,k\in \N,\;\forall\,t\in [0,\infty),
\end{aligned}
\]
and the pointwise convergence trivially follows by the weak* convergence of the $u_k$. Thus, equi-coercivity w.r.t. $x$ is proved.

It can also be seen that the running cost integrand trivially satisfies all the hypotheses of~\cite[Theorem 3.1]{DellaFreddi25}. 
Thus, such family of optimal control problems is pattern-preserving  without being dissipative, somehow providing a ``boundary'' to our sufficient conditions provided in Theorem~\ref{thm:Main}.


 To complete our discussion, we apply  pattern-preservation to find an optimal control for the infinite-horizon problem arising from a preliminary   characterization the finite-horizon ones.

To this aim, let us identify all (candidate) optimal control-state pairs, for a given $T>0$ (finite), via Pontryagin principle. 
These candidates are referred to as Pontryagin extremals.
As in Example~\ref{example:Dissipative1}, since any solution is strictly positive and $U= [0,2]$, we may omit the absolute value signs in the running cost when formulating the conditions given by the Pontryagin principle.
The Hamiltonian $H:\R^4\to \R$ is thus defined by 
\begin{equation*}
    \begin{aligned}
        H(t,u,x,p) &= \left(\frac{u}{3} + x \right) e^{-2t} + p (1-u)x +\chi_U(u)
\\&= x e^{-2t} + p x + u \left( \frac{1}{3} e^{-2t} - p x \right)+\chi_U(u).
    \end{aligned}
\end{equation*}
Given $(u,x,p):[0,T]\to \R$ a control-state-costate triple satisfying the condition of Pontryagin principle, we have that the costate $p$ satisfies
\begin{equation}
\begin{cases}
p'(t) &=  -e^{-2t} - p(t)(1-u(t)), \\
 p(T) &=0.
\end{cases}
\end{equation}
By the Weirstrass condition, the optimal control satisfies:
\[
u(t) = 
\begin{cases}
2, & \text{if }\phi(t) < 0,\\
0, & \text{if }\phi(t) > 0,\\
\end{cases}
\]
where  the \emph{switching function} is defined by
$\phi(t) := \frac{1}{3} e^{-2t} - p(t)x(t)$.
\\
If $\phi(t) = 0$ we cannot infer the value of $u$. Let us now split the reasoning in various intermediate steps.\\
\emph{Step 1:  Characterizing singular arcs.}
Suppose  that there exists a non-trivial interval $I\subset [0,T]$ such that
$
\phi(t)=0$ for all  $t\in I$.
In such situation is common to say that in $I$ we have a \emph{singular arc}.
Computing the derivative of $\phi$ we have
\begin{equation}\label{eq:derivative}
\hspace{-0.05cm}\phi'(t)\hspace{-0.07cm}=\hspace{-0.07cm}-\tfrac{2}{3}e^{-2t}-p'(t)x(t)-p(t)x'(t)=e^{-2t} (x(t)-\tfrac{2}{3}).
\end{equation}
Thus, we must have $x(t)=\frac{2}{3}$ for all $t\in I$. This also implies that $x'(t)=0$ for all $t\in I$, which in turn, yields to $u(t)=1$  for all $t\in I$ (since $x$ cannot reach $0$ in finite time).
Moreover, substituting these values in $\phi(t)$ we obtain
\[
0=\frac{1}{3} e^{-2t} - p(t)x(t)=\frac{1}{3} e^{-2t} - p(t)\frac{2}{3}\;\;\;\;\;\forall \,t\in I,
\]
which implies
$p(t)=\frac{1}{2}e^{-2t}$ for all $t\in I$.
Note that this is compatible with the adjoint equation, since 
$p'(t)=-e^{-2t}$.
Summarizing, along a singular arc we have
\begin{equation}\label{eq:SingularArcCara}
x(t)=\frac{2}{3},\;\;\;u(t)=1,\;\;\;p(t)=\frac{1}{2}e^{-2t},\;\;\;\forall\,t\in I.
\end{equation}
\emph{Step 2: Structure of singular-arc candidates.}
Let us suppose that a singular arc exists, and call it $I=(t_1,t_2)$ with $t_1<t_2$. Since $x(0)=1$ and $p(T)=0$,  by~\eqref{eq:SingularArcCara} and the definition of  $\phi$ we necessarily have that $(t_1,t_2)\subset (0,T)$.
 Now, let us consider $I^+=(t_2,\tau)\subset [0,T]$ the maximal interval of this form in which $\phi$ does not change sign. 
Suppose first that $\phi>0$ in $I^+$. This implies that $u(t)=0$ and $x(t)=\frac{2}{3}e^{t-t_2}>\frac{2}{3}$ in $I^+$. This, by~\eqref{eq:derivative} implies that $\phi'>0$ in $I^+$. Thus, $\phi$ cannot change sign and then $\tau=T$.

Otherwise, if $\phi<0$ in $I^+$, this implies that $u(t)=2$ and $x(t)=\frac{2}{3}e^{-(t-t_2)}<\frac{2}{3}$ in $I^+$. This implies that $\phi$ is strictly decreasing in $I^+$ and thus cannot change sign in $I^+$ and again $\tau=T$. On the other hand this is not compatible with the fact that $\phi(T)>0$, ruling out this alternative.

Consider now $I^-=(\tau,t_1)\subset [0,t_1)$ the maximal interval in which $\phi$ does not change sing. 
Let us suppose that first $\phi<0$ in $I^-$; again, this implies that $u(t)=2$ and $x(t)=\frac{2}{3}e^{-(t-t_1)}>\frac{2}{3}$ in $I^-$. This implies that $\phi$ is strictly increasing in $I^-$, thus cannot change sign in $I^-$ and, hence, $\tau=0$. Note that this is compatible with $x(0)=1$.
If instead $\phi>0$ in $I^-$, then $u(t)=0$ and  $x(t)=\frac{2}{3}e^{(t-t_1)}<\frac{2}{3}$ in $I^-$. Therefore,  $\phi$ is strictly decreasing in $I^-$ and, thus,  cannot change sign in $I^-$, implying $\tau=0$. On the other hand note that this is incompatible with $x(0)=1$, and also such possibility is ruled out.
Summarizing, if a singular arc exists, the solution is necessarily of the following form
\begin{equation}\label{eq:singularSolStructure}
u(t)=\begin{cases}
2&\text{if }t\in (0,t_1),\\
1& \text{if }t\in (t_1,t_2),\\
0& \text{if }t\in (t_2,T),
\end{cases}
\end{equation}
for some $0<t_1<t_2<T$.\\
\emph{Step 3: Non-singular candidates: at most $1$ switching point.} Suppose that $(x,u,p)$ satisfies the Pontryagin principle conditions and does not exhibit  singular arcs.
Suppose by contradiction that there are $2$ non-trivial switching points $0< t_1<t_2<T$.
In $(t_2,T]$ we have $\phi(t)>0$ (since $\phi(T)>0)$ and thus $u(t)=0$.
In $(t_1,t_2)$ we have $\phi(t)<0$ and thus $u(t)=2$.
Then, in $(t_1,t_2)$ $x$ is strictly decreasing,  because $x'$ is strictly negative. Since $\phi$ is $\cC^1$ in $[t_1,t_2]$  (by continuity of $x$ and~\eqref{eq:derivative}) and takes the same value at the endpoints, by Rolle's theorem  there exists $t_3\in (t_1,t_2)$ such that $\phi'(t_3)=0$, i.e.,  $x(t_3)=\frac{2}{3}$. But then, since $x$ is decreasing in $(t_1,t_2)$ we have that
\[
x(t)>\frac{2}{3} \;\;\;\text{in } (t_1,t_3)\;\;\;\text{ and }\;\;\;
x(t)<\frac{2}{3} \;\;\;\text{in } (t_3,t_2)
\]
and then $\phi'(t)>0$ if $t\in (t_1,t_3)$ and
$\phi'(t)<0$ if $t\in(t_3,t_2)$, in contradiction with the fact that $\phi(t_1)=\phi(t_2)=0$ and $\phi(t)<0$ in $(t_1,t_2)$.\\
\emph{Step 4: Structure of non-singular candidates.}
Suppose by contradiction that there are no switching points. This implies that $u(t)\equiv 0$ for all $t\in [0,T]$ (since we already noted that $\phi(T)>0$). 
Computing we have $\displaystyle J_T(u)=\int_0^T (0+e^t)e^{-2t}dt=1-e^{-T}
$
while for $v\equiv 2$ we have
$\displaystyle 
J_T(v)=\int_0^T(\frac{2}{3}+e^{-t})e^{-2t} dt=\frac{-e^{-3T}-e^{-2T}+2}{3}
$.
    Now, $J_T(v)-J_T(u)=\frac{-3+3e^{-T}-e^{-3T}-e^{-2T}+2}{3}<0$ for all $T>1$, contradicting the optimality of $u$.
Then any Pontryagin extremal with no singular arcs $u$ is necessarily of the form
\begin{equation}\label{eq:NonsingularSolStructure}
u(t)=\begin{cases}
2 &\text{if } t\in [0,t^T_s)\\
0 &\text{if } t\in (t^T_s,T)\\
\end{cases}
\end{equation}
for a certain $t_s^T\in (0,T)$, at least for $T>1$.\\
\emph{Fact 5: The switching times.}
Let us suppose that $T>1$ and consider first a non-singular solution $u$ of the form~\eqref{eq:NonsingularSolStructure}.
Since $u=0$ for $t \ge t_s$ we have
$x'(t) = x(t)$ which implies $x(t) = x(t_s) e^{t-t_s}$.
Moreover, since $p(T)=0$ we can compute $p(t)$ in such interval, obtaining
$
 p(t) = e^{-2t} - e^{-(T+t)}$.
Substituting in the switching condition $\phi(t_s)=0$ we have
\[
\frac{1}{3} e^{-2 t_s} = p(t_s)x(t_s) = \big(e^{-2 t_s} - e^{-(T+t_s)}\big) x(t_s).
\]
On $[0,t_s]$, $u=2 \implies x(t_s)=e^{-t_s}$, so
\[
\frac{1}{3} e^{-2 t_s} = \big(e^{-2 t_s} - e^{-(T+t_s)}\big) e^{-t_s} \implies e^{t_s} = 3 \big(1 - e^{-(T-t_s)}\big).
\]
and then
$t_s^T=\ln\left (\frac{3}{1+3e^{-T}}\right)$.\\
Summarizing, the closed form of non-singular Pontryagin extremal is 
\[
u_{ns,T}(t)=\begin{cases}
2\;\;&\text{if } t\in \left (0,\ln\left (\frac{3}{1+3e^{-T}}\right)\right)\\
0\;\;&\text{if } t\in \left (\ln\left (\frac{3}{1+3e^{-T}}\right), T\right).
\end{cases}
\]
Let us now characterize the times $t_1<t_2$ for the singular-arc candidate.
Since $x(t_1)=\frac{2}{3}$ (recall equation~\eqref{eq:SingularArcCara} in \emph{Step 1})  and $x(t)=e^{-t}$ in $[0, t_1)$  we have that 
$
\frac{2}{3}=e^{-t_1}$
and thus $t_1=\ln(\frac{3}{2})$.
For $t_2$ we have that $p(t)=\frac{1}{2}e^{-2t}$ in $(t_1,t_2)$ and $p(t)=e^{-2t}-e^{-(T+t)}$ in $(t_2,T)$ and then
$
\frac{1}{2}e^{-2t_2}=e^{-2t_2}-e^{-(T+t_2)}
$
which implies 
$
\frac{1}{2}e^{-2t_2}=e^{-T}e^{-t_2}
$
and thus
$
\frac{1}{2}=e^{-T}e^{t_2}.
$
In turn, this implies that $
t_2(T)=T-\ln(2)$.
Thus, the singular Pontryagin extremals are of the form 
\[
u_{s,T}=\begin{cases}
2\;\;\;&\text{if }t\in (0,\ln(\frac{3}{2}))\\
1\;\;\;&\text{if }t\in (\ln(\frac{3}{2})), T-\ln(2))\\
0\;\;\;&\text{if }t\in ( T-\ln(2),T).
\end{cases}
\]
\emph{Step 6: The finite-horizon optimal control.}
By numerically comparing the costs associated with $u_{s,T}$ and $u_{ns,T}$, it is observed that, for sufficiently large $T$, the control $u_{s,T}$ yields a lower cost and is therefore optimal. For brevity, the explicit computations are omitted. The optimality of $u_{s,T}$ for large $T$ can be further corroborated using the open-source toolbox for optimal control problems~\texttt{BOCOP} ( see~\cite{Bocop,BocopExamples}).
Figure~\ref{figure:Extremal} illustrates the control $u_{s,T}$ for $T=10$, providing a clear depiction of its qualitative behavior.
 \\
\emph{Step 7: The infinite-horizon optimal control.}
Let us find the limiting control: since $t_1\equiv \ln(\frac{3}{2})$   and $t_2(T)\to \infty$ as $T\to \infty$ 
we have that $u_{s,T} \to u_{s,\infty}$ in $\cU$ where \begin{equation}\label{eq:Infinite-HorizonControl}
u_{s, \infty}(t)=\begin{cases}
2\;\;&\text{if }t\in (0,\ln(\frac{3}{2})),\\
1\;\;&\text{if }t\in (\ln(\frac{3}{2}), \infty).
\end{cases}
\end{equation}
By applying the pattern-preserving property,
we can thus conclude that $u_{s, \infty}$ is an optimal control for the infinite-horizon problem.\hfill $\triangle$ 
\end{example}
Note that both these examples exhibit a typical \emph{turnpike behavior} (as introduced and discussed in~\cite{FaulGrune22, Faulwasser21,Gruene2022,Hara2023,Trelat23, Angeli2024}). Indeed, focusing on the second example for conciseness, the optimal finite-horizon solutions remain ``most of the time'' close to the equilibrium $\hat z = (\hat u, \hat x)$, defined by $\hat u := 1$ and $\hat x := \tfrac{2}{3}$, for all finite times $T$. A similar property holds for the infinite-horizon control, which appears to stabilize in finite time at the equilibrium $\hat z$ and subsequently remains there, see also Figure~\ref{figure:Extremal}.
By adapting the results in~\cite[Theorem 14]{Faulwasser21} for time-dependent running costs, this turnpike/convergence to $\hat z$ behavior in turn implies that the system satisfies a \emph{local} strict dissipativity condition in a (functional) neighborhood of this equilibrium $\hat z$. This is consistent with our analysis, which rules out the existence of a \emph{global} coercive storage function. We further emphasize that, in contrast to classical turnpike-based approaches, our framework does not require any a priori knowledge of such optimal ``operating point'' $\hat z$ in order to establish the pattern-preserving property. This thus allows us to conclude that infinite-horizon optimal controls can be obtained as limits of finite-horizon optimal controls without any further information on ``candidate equilibria''.
Summarizing, this simple example also highlights both the relation and the distinction between our sufficient conditions and those arising in dissipativity-based turnpike analysis~\cite{FaulGrune22, Faulwasser21,Gruene2022,Hara2023,Trelat23, Angeli2024}, which substantially focus on distinct behavior of infinite optimal control problems.
\begin{figure}[t!]\label{figure:Extremal}
\centering
\begin{tikzpicture}[scale=0.75]

\draw[->] (0,0) -- (10.5,0) node[right] {$t$};
\draw[->] (0,0) -- (0,2.5) node[above] {$u_{s,T}(t)$};

\draw[very thick, blue] (0,2) -- (0.4055,2);
\draw[very thick, blue] (0.4055,1) -- (9.3069,1);
\draw[very thick, blue] (9.3069,0) -- (10,0);

\draw[dashed,red] (0.4055,0) -- (0.4055,2);
\draw[dashed,red] (9.3069,0) -- (9.3069,1);

\node[below] at (0.4055,-0.1) {$\ln(3/2)$};
\node[below] at (8.9,-0.1) {$T-\ln 2$};
\node[below] at (10.2,-0.1) {$T$};
\draw (0.4055,0) -- (0.4055,-0.1);
\draw (9.3069,0) -- (9.3069,-0.1);
\draw (10,0.2) -- (10,-0.2);
\node[left] at (0,2) {$2$};
\node[left] at (0,1) {$1$};
\node[left] at (0,0) {$0$};

\end{tikzpicture}
\caption{Graphical representation of the singular-arc Pontryagin extremal $u_{s,T}$ in Example~\ref{ex:NONDISSIpative} for $T=10$. The \texttt{BOCOP} toolbox can be used to numerically observe that these extremals are optimal, at least for large 
$T$, and this can be formally confirmed by explicit computation of the cost. The infinite-horizon control $u_{s, \infty}$ is obtained as limit, when $T\to \infty$, of such piecewise-constant functions, and its formal definition is in~\eqref{eq:Infinite-HorizonControl}. }
\end{figure}
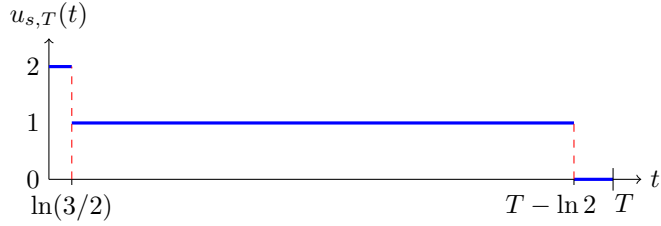

\subsection{(Non-)Linear Quadratic Regulator}\label{subsec:LQR}
A well-known setting in which the \emph{pattern-preserving property holds} is the classical case of LQR optimal control problems. In this setting a linear state-equation of the form $x'=Ax+Bu$ is coupled with a positive-definite quadratic running cost of the form $\ell(x,u)=x^\top Qx+u^\top R u$ for some matrices $Q\succeq 0$ and $R\succ 0$. Under stabilizability conditions, it can be proved that the sequence of finite-horizon optimal controls converges weakly in $L^2$ to an infinite-horizon optimal control.  This is illustrated for example in~\cite[Chapter 6]{Liberzon12} or~\cite[Chapter 8]{magni2014advanced}, in which also several references are provided.

In this short subsection we apply Theorem~\ref{thm:Main} to a possible non-linear and state-constraint extension of such problem, known as \emph{(non-linear)-quadratic regulator problems}, (n-QR).
Let us consider the problems, parameterized by $T\in (0,+\infty]$ defined as in~\eqref{eq:RiscritturaProblem} with a  running cost function $\ell:(0,\infty)\times \R^n\times \R^m\to [0,\infty]$  defined by
\[
\ell(t,x,u)=\ell_{QR}(t,x,u):=\ell_1(t,x)+u^\top Ru
\]
where $\ell_1:(0,\infty)\times \R^n\to [0,+\infty]$ is a nonnegative normal  integrand and $R\in \R^{m\times m}$, $R\succ 0$. 
Since $\ell_1$ is possibly infinite valued, it can incorporate a state-constraint of the form $x(t)\in X$ for  a closed $X\subset \R^n$, for example if $\ell_1$ is of the form $\ell_1(t,x)=x^\top Q x+\chi_X(x)$, with $Q\in \R^{n\times n}$, $Q\succeq 0$. The cost functional $\cJ_T:\cU\times \cX\to [0,\infty]$ thus reads
\[
\cJ_T(u,x)=\int_0^T \ell_1(t,x(t))+u^\top(t) Ru(t)\;dt.
\]
The natural choice of the control space is thus $\cU=L^2(\R_+,\R^m)$ with its weak$^*$ (equivalently, weak) topology, while the state space is set to be, as always, $\cX=\Wl^{1,1}([0,\infty),\R^n)$ equipped  with the topology of uniform convergence on compact intervals.
We define the admissible sets $\cA_T\subset \cU\times \cX$ via the autonomous control-affine differential Cauchy problem
\begin{equation}\label{eq:LQRproblemEquation}
\begin{cases}
x'(t)=a(t,x(t))+b(t,x(t))u(t)\\
x(0)=x_0,
\end{cases}
\end{equation}
where the functions $a$ and $b$ satisfies Assumption~\ref{assumpt:Basic0_L1} in the stronger form required by condition 1. of Theorem~\ref{thm:Main}.  

For $T\in (0,\infty]$, we denote the optimal control problem induced by the functional $\cF_{T}:=\cJ_T+\chi_{\cA_T}$ (with $\cA_T$ defined as in~\eqref{eq:AdmissibleSetA_T}) by the symbol $\LQR{T}{x_0}$, in which it is also explicit the dependence from the initial condition $x_0\in \R^n$.

We now show that under the dissipativity assumption, the pattern-preserving property directly applies to this class of problems, giving an approximation method to find infinite-horizon controls.
\begin{prop}\label{prop:nQRopt}
Suppose that $a$ and $b$ in~\eqref{eq:LQRproblemEquation} satisfies the hypothesis \emph{1.} of Theorem~\ref{thm:Main} and the system in~\eqref{eq:LQRproblemEquation} is dissipative with respect to $\ell_{QR}$ with a \emph{coercive} storage function.
Then, given any $x_0\in \R^n$ we have two alternatives:
\begin{enumerate}
    \item[(A)] $\LQR{\infty}{x_0}\equiv +\infty$; \emph{or}
\item[(B)] the sequence of problems $\LQR{T_k}{x_0}$ is pattern preserving. More specifically,  given any sequence $0<T_k\to \infty$, for any $k\in \N$ the problem $\LQR{T_k}{x_0}$ admits an optimal solution, and every sequence of optimal controls $u_{T_k}$ of $\LQR{T_k}{x_0}$ admits a subsequence converging (weakly in $L^2$) to an optimal control for the problem $\LQR{\infty}{x_0}$.
\end{enumerate}
\end{prop}
\begin{proof}
Given $x_0\in \R^n$, let us suppose that there exists at least a feasible control $\bar u\in L^2((0,\infty),\R^m)$ for  $\LQR{\infty}{x_0}$ otherwise we are in case \emph{(A)} and there is nothing to prove. This implies that
\begin{equation}\label{eq:FIniteCost}
\int_0^\infty \ell_1(t,x^{\bar u,x_0}(t))\, +\bar u^\top(t)R \bar u(t)\;dt=M
\end{equation}
for a certain $M\geq 0$.
This also implies that $\cA_T\neq \varnothing$ for all $T\in (0,\infty)$. Indeed,  since the couple $(\bar u,x^{\bar u,x_0})\in \cA_\infty$, after the already performed tail-replacing as in~\eqref{eq:TailReplacement} with $u_f\equiv0$, it generates an element of $\cA_T$. This implies that (at least) an optimal solution exists for $\LQR{T}{x_0}$, for any $T\in (0,\infty)$, see for example~\cite[Theorem 23.11]{Clarke13}.

Let us consider sequences $0\leq T_k\to \infty$ and $u_k\in L^2((0,\infty),\R^m)$, where $u_k$ is an optimal control for the problem $\LQR{T_k}{x_0}$ (with  the aforementioned tail-replacement $u(t)=0$ for a.a. $t\geq T_k$).
Since $R\succ 0$, there exists $a>0$ such that $I_m\preceq aR$.
We then have
\[
\begin{aligned}
\|u_k\|^2_2&=\int_0^\infty u^2_k(t)\,dt=\int_0^{T_k}u^2_k(t)\,dt \leq a\int_0^{T_k}u^\top_k(t)Ru_k(t)\,dt
\\&\leq a\int_0^{T_k} \ell_1(t,x^{u_k,x_0}(t))+u^\top_k(t)Ru_k(t)\,dt.
\end{aligned}
\]
Since any $u_k$ is an optimal control for $\LQR{T_k}{x_0}$ we also have
\[
\begin{aligned}
\|u_k\|^2_2&\leq a\int_0^{T_k} \ell_1(t,x^{\bar u,x_0}(t))+\bar u^\top(t)R\bar u(t)\,dt \\&\leq a\int_0^{\infty} \ell_1(t,x^{\bar u,x_0}(t))+\bar u^\top(t)R\bar u(t)\,dt\leq aM
\end{aligned}
\]
recalling equation~\eqref{eq:FIniteCost}.
This implies that $u_k$ is bounded in $L^2((0,\infty),\R^m)$ and thus it admits a weakly converging subsequence, let us say $u_{k_h}\to u_\infty$, for a certain $u_\infty\in L^2((0,\infty),\R^m)$.
Now, it is easy to see that all the hypothesis of Theorem~\ref{thm:Main} are satisfied: the system is dissipative by hypothesis, while the other conditions hold by defining $\ell_2(u)=u^\top R u$, and taking $u_f=0$.  We can thus conclude that $\LQR{T_k}{x_0}$ is a pattern-preserving sequence of problems and then  $u_\infty$ is an optimal control for  $\LQR{\infty}{x_0}$, concluding the proof.
\end{proof}
Let us note that, in this case, the differential characterization of dissipativity provided in~\eqref{eq:DifferentialCoercivity} turns out to be
\[
\begin{aligned}
    \inp{\nabla S(x)}{a(t,x)+b(t,x)u}\leq \ell_1(t,x)+u^\top R u,\;\; \\ \forall (t,x,u)\in \R_+\times \R^n\times \R^m\;\;\text{such that }\ell_1(t,x)<\infty. 
\end{aligned}
\]
If $a,b$ are time-independent polynomial, $\ell_1(t,x)=q(x)+\chi_X(x)$, with $q$ a polynomial and $X\subseteq \R^n$ a closed set, this differential inequality can be solved via SOS or semidefinite optimization techniques (see for example~\cite{caltechthesis}). If with such numerical methods we are able to find a coercive (SOS) polynomial function $S:\R^n \to \R$ satisfying such inequality we already have a certificate for pattern-preserving property.

\section{Conclusions}\label{sec:Conclusions}
In this work, we developed a variational perspective to compare finite-horizon and infinite-horizon optimal control problems in the presence of dissipative dynamics. Our main result showed that dissipativity, complemented by a coercive storage function, guarantees that the asymptotic behavior of optimal controls is well captured by the sequence of finite-horizon solutions, so that the infinite-horizon optimal control emerges as their limit. In this sense, dissipativity provides a structural mechanism ensuring the persistence of qualitative features of optimal controls as the time horizon grows.

The proposed framework connects classical dissipativity theory with tools of variational convergence (notably, $\Gamma$-convergence), offering a unified viewpoint that is both theoretically robust and practically verifiable. The validity and sharpness of the sufficient conditions were demonstrated through several numerical examples, which also clarify the boundaries of the approach.

A natural direction for future research is to further explore the role of (strict) dissipativity in infinite-horizon optimal control problems, further exploiting the variational framework introduced here and to extend the class of systems for which pattern preservation can be guaranteed.

\

\noindent
{\bf Acknowledgements.}  
M.D.R.\ and L.F.\ are members of GNAMPA--INdAM. T.A.L. was supported by the CNPq (Conselho Nacional de Desenvolvimento Científico e Tecnológico) grant number 443674/2024-8.

\bibliography{bibnol}  

\begin{thebibliography}{10}

\bibitem{Halkin74}
H.~Halkin, ``Necessary conditions for optimal control problems with infinite
  horizons,'' {\em Econometrica}, vol.~42, no.~2, pp.~267--272, 1974.

\bibitem{BasCassFra18}
V.~Basco, P.~Cannarsa, and H.~Frankowska, ``Necessary conditions for infinite
  horizon optimal control problems with state constraints,'' {\em Math. Control
  Relat. Fields}, vol.~8, no.~3-4, pp.~535--555, 2018.

\bibitem{CANNFra18}
P.~Cannarsa and H.~Frankowska, ``Value function, relaxation, and transversality
  conditions in infinite horizon optimal control,'' {\em J. Math. Anal. Appl.},
  vol.~457, no.~2, pp.~1188--1217, 2018.

\bibitem{Faulwasser21}
T.~Faulwasser and C.~M. Kellett, ``On continuous-time infinite horizon optimal
  control-{Dissipativity}, stability, and transversality,'' {\em Automatica},
  vol.~134, p.~109907, 2021.

\bibitem{Vinter2010Book}
R.~B. Vinter, {\em Optimal Control}.
\newblock Modern Birkh{\"a}user Classics, Birkh{\"a}user Boston, 2010.

\bibitem{Clarke13}
F.~Clarke, {\em Functional Analysis, Calculus of Variations and Optimal
  Control}.
\newblock Graduate Texts in Mathematics, Springer London, 2013.

\bibitem{Liberzon12}
D.~Liberzon, {\em Calculus of Variations and Optimal Control Theory: A Concise
  Introduction}.
\newblock Princeton, NJ, USA: Princeton University Press, 2012.

\bibitem{DellaFreddi25}
M.~{Della Rossa} and L.~Freddi, ``Pattern-preserving optimal control problems
  with increasing time horizon,'' {\em ESAIM: COCV}, vol.~31, p.~102, 2025.

\bibitem{DalMaso93}
G.~Dal~Maso, {\em An introduction to {$\Gamma$}-convergence}, vol.~8 of {\em
  Progress in Nonlinear Differential Equations and their Applications}.
\newblock Birkh\"auser Boston, Inc., Boston, MA, 1993.

\bibitem{Braides02}
A.~Braides, {\em $\Gamma$-Convergence for Beginners}.
\newblock Oxford University Press, 2002.

\bibitem{ButtDalMas82}
G.~Buttazzo and G.~Dal~Maso, ``{$\Gamma $}-convergence and optimal control
  problems,'' {\em J. Optim. Theory Appl.}, vol.~38, no.~3, pp.~385--407, 1982.

\bibitem{BelButFre93}
M.~Belloni, G.~Buttazzo, and L.~Freddi, ``Completion by {Gamma}-convergence for
  optimal control problems,'' {\em Ann. Fac. Sci. Toulouse Math. (6)}, vol.~2,
  no.~2, pp.~149--162, 1993.

\bibitem{Willems73}
J.~C. Willems, ``Dissipative dynamical systems part {I}: General theory,'' {\em
  Archive for Rational Mechanics and Analysis}, vol.~45, no.~5, pp.~321--351,
  1972.

\bibitem{Willems1972b}
J.~C. Willems, ``Dissipative dynamical systems part {II}: Linear systems with
  quadratic supply rates,'' {\em Archive for Rational Mechanics and Analysis},
  vol.~45, no.~4, pp.~352--393, 1972.

\bibitem{BorMash06}
B.~Brogliato, B.~Maschke, R.~Lozano, and O.~Egeland, {\em Dissipative Systems
  Analysis and Control: Theory and Applications}.
\newblock Communications and Control Engineering, Springer, 3rd~ed., 2020.

\bibitem{FaulKorda17}
T.~Faulwasser, M.~Korda, C.~N. Jones, and D.~Bonvin, ``On turnpike and
  dissipativity properties of continuous-time optimal control problems,'' {\em
  Automatica}, vol.~81, pp.~297--304, 2017.

\bibitem{Gruene2022}
L.~Gr{\"u}ne, ``Dissipativity and optimal control: Examining the turnpike
  phenomenon,'' {\em IEEE Control Systems Magazine}, vol.~42, no.~2,
  pp.~74--87, 2022.

\bibitem{Hara2023}
K.~Hara, M.~Inoue, and N.~Sebe, ``Dissipativity-constrained learning of {MPC}
  with guaranteeing closed-loop stability,'' {\em Automatica}, vol.~157,
  p.~111271, 2023.

\bibitem{Trelat23}
E.~Trelat, ``Linear turnpike theorem,'' {\em Mathematics of Control, Signals,
  and Systems}, vol.~35, 2023.

\bibitem{Angeli2024}
D.~Angeli and L.~Gr{\"u}ne, ``Dissipativity in infinite horizon optimal control
  and dynamic programming,'' {\em Applied Mathematics \& Optimization},
  vol.~89, 2024.

\bibitem{FaulGrune22}
T.~Faulwasser and L.~Gr\"une, ``Chapter 11 {-} {Turnpike} properties in optimal
  control: An overview of discrete-time and continuous-time results,'' in {\em
  Numerical Control: Part A} (E.~Trélat and E.~Zuazua, eds.), vol.~23 of {\em
  Handbook of Numerical Analysis}, pp.~367--400, Elsevier, 2022.

\bibitem{Buttazzo89}
G.~Buttazzo, {\em Semicontinuity, Relaxation and Integral Representation in the
  Calculus of Variations}, vol.~207 of {\em Pitman Research Notes in
  Mathematics Series}.
\newblock Longman Scientific \& Technical, Harlow; copublished in the US with
  John Wiley \& Sons, Inc., New York, 1989.

\bibitem{Evans24}
L.~Evans, ``An introduction to mathematical optimal control theory,'' 2024.
\newblock Lecture Notes. Available at:
  \url{https://math.berkeley.edu/~evans/control.course.pdf}.

\bibitem{Hale}
J.~K. Hale, {\em Ordinary Differential Equations}.
\newblock Dover Publications, 1997.

\bibitem{VanDerSchaft17}
A.~{van der Schaft}, {\em L2-Gain and Passivity Techniques in Nonlinear
  Control}.
\newblock Communications and Control Engineering, Springer, 3rd~ed., 2017.

\bibitem{BEFB:94}
S.~Boyd, L.~El{ }Ghaoui, E.~Feron, and V.~Balakrishnan, {\em Linear Matrix
  Inequalities in System and Control Theory}.
\newblock Philadelphia, PA: SIAM Studies in Applied Mathematics, 1994.

\bibitem{sostools}
A.~Papachristodoulou, J.~Anderson, G.~Valmorbida, S.~Prajna, P.~Seiler, P.~A.
  Parrilo, M.~M. Peet, and D.~Jagt, {\em {SOSTOOLS}: Sum of squares
  optimization toolbox for {MATLAB}}.
\newblock \texttt{http://arxiv.org/abs/1310.4716}, 2021.

\bibitem{caltechthesis}
P.~A. Parrilo, {\em {Structured} semidefinite programs and semialgebraic
  geometry methods in robustness and optimization}.
\newblock California Institute of Technology, 2000.

\bibitem{Arcak2016}
M.~Arcak, C.~Meissen, and A.~Packard, {\em Networks of Dissipative Systems}.
\newblock Springer International Publishing, 2016.

\bibitem{FL07}
I.~Fonseca and G.~Leoni, {\em Modern Methods in the Calculus of Variations:
  {$L^p$} Spaces}.
\newblock Springer Monographs in Mathematics, Springer, New York, 2007.

\bibitem{Bocop}
I.~S. Team~Commands, ``Bocop: an open source toolbox for optimal control.''
  \url{http://bocop.org}, 2017.

\bibitem{BocopExamples}
F.~J. Bonnans, D.~Giorgi, V.~Grelard, B.~Heymann, S.~Maindrault, P.~Martinon,
  O.~Tissot, and J.~Liu, ``{Bocop – A collection of examples},'' tech. rep.,
  INRIA, 2017.

\bibitem{magni2014advanced}
L.~Magni and R.~Scattolini, {\em Advanced and Multivariable Control}.
\newblock Pitagora, 2014.

\bibitem{AS1976}
A.~Ambrosetti and C.~Sbordone, ``{$\Gamma \sp{-}$}-convergenza e
  {$G$}-convergenza per problemi non lineari di tipo ellittico,'' {\em Boll.
  Un. Mat. Ital. A (5)}, vol.~13, no.~2, pp.~352--362, 1976.

\end{thebibliography}
\bibliographystyle{ieeetr}
\newpage
    \appendix
\section{$\Gamma$-convergence}\label{sec:Appendix}

In this short appendix we recall in a condensed form the definition and an useful sufficient condition for sequential $\Gamma$-convergence for sequences of functionals defined on a topological space $\cU\times \cX$.
\\
Since the pioneering work of Buttazzo and Dal Maso  (see~\cite{DalMaso93, ButtDalMas82}), an 
 appropriate notion of variational convergence  is provided by  sequential $\Gamma$-limits, whose definitions for a general sequence $\cF_k:\cU\times \cX\to \overline \R$ are
$$
\begin{array}{l}
\hspace{-0.18cm}\displaystyle\G^-_\seq(\cU\hspace{-0.07cm}\times\hspace{-0.07cm}\cX)\liminf_\koo \cF_k(u,x)\hspace{-0.07cm}:=\hspace{-0.07cm}\inf_{u_k\to u}\inf_{x_k\to x}\liminf_\koo
\cF_k(u_k,x_k),\\
\hspace{-0.18cm}\displaystyle\G^-_\seq(\cU\hspace{-0.07cm}\times\hspace{-0.07cm}\cX)\limsup_\koo \cF_k(u,x)\hspace{-0.07cm}:=\hspace{-0.15cm}\inf_{u_k\to u}\inf_{x_k\to x}\limsup_\koo
\cF_k(u_k,x_k).
\end{array}
$$  When they coincide, we denote their common value by writing
$$
\G_\seq^-(\cU\times\cX)\lim_\koo \cF_k.
$$
 \begin{remark}\label{rem_lirs}
It is easy to check that a sufficient conditions implying $\G^-_\seq(\cU\times\cX)$-convergence of $\cF_k$ to $\cF$ is given by 
\begin{enumerate}
\item (\emph{liminf inequality}) for all sequences $(u_k,x_k)\to (u,x)$ we have $\dis \cF(u,x)\leq \liminf_{k\to \infty} \cF_k(u_k,x_k)$;
\item (\emph{recovery sequence}) there exists a sequence $(u_k,x_k)\to (u,x)$ such that  $\dis \cF(u,x)\geq \limsup_{k\to \infty} \cF_k(u_k,x_k)$.
\end{enumerate}
Clearly 1.\ is also necessary.  On the other hand, the necessity of 2., that is the equivalence (between 1.-2.\ and $\G_\seq^-$-convergence as well as between topological $\G$-convergence  and $\G_\seq$ convergence) 
holds in first countable spaces (see \cite[Proposition 8.1]{DalMaso93} or for equi-coercive sequences in spaces in which relatively compact sets are metrizable like, e.g., Banach spaces with a separable dual space, or dual of a separable Banach space, with the weak, respectively weak*, topology (see  \cite{AS1976} or \cite{DalMaso93}). In the setting considered in this manuscript, this is indeed our case, therefore conditions 1.\  and 2.\  are both necessary and sufficient for $\G_\seq$-convergence.
\end{remark}

\end{document}